%% file: main.tex
\documentclass{article}
\usepackage[utf8]{inputenc}

\usepackage{amsmath}
\usepackage{amssymb}
\usepackage{amsfonts}
\usepackage{amsrefs}
\usepackage{amsthm}
\usepackage{mathtools}
\usepackage{multirow}
\usepackage{enumitem}
\usepackage{tikz-cd}
\usepackage{soul}
\usepackage{mathrsfs}
\usepackage{bbm}
\usepackage{float}
\usepackage{authblk}
\usepackage[labelfont=bf, labelsep=none]{caption}
\usepackage{abstract}

\newcommand{\g}{\mathfrak{g}}
\newcommand{\gaff}{\mathfrak{\hat{g}}}
\newcommand{\Uq}{U_{q}(\mathfrak{g})}
\newcommand{\Uaff}{U_{q}(\mathfrak{\hat{g}})}
\newcommand{\Utor}{U_{q}(\mathfrak{g}_{\mathrm{tor}})}
\newcommand{\UtorA}{U_{q}(\mathfrak{sl}_{n+1,\mathrm{tor}})}
\newcommand{\Uh}{\mathcal{U}_{h}}
\newcommand{\Uv}{\mathcal{U}_{v}}
\newcommand{\B}{\mathcal{B}}
\newcommand{\Bd}{\mathcal{\dot{B}}}
\newcommand{\Bdd}{\mathcal{\ddot{B}}}
\newcommand{\Bh}{\mathcal{B}_{h}}
\newcommand{\Bv}{\mathcal{B}_{v}}
\newcommand{\e}{\mathfrak{e}}
\newcommand{\te}{\mathfrak{t}}
\newcommand{\xip}{x_{i}^{+}}
\newcommand{\xim}{x_{i}^{-}}
\newcommand{\xipm}{x_{i}^{\pm}}
\newcommand{\xjp}{x_{j}^{+}}
\newcommand{\xjm}{x_{j}^{-}}
\newcommand{\xjpm}{x_{j}^{\pm}}
\newcommand{\xp}{x^{+}}
\newcommand{\xm}{x^{-}}
\newcommand{\xpm}{x^{\pm}}
\newcommand{\Qov}{\mathring{Q}^{\vee}}
\newcommand{\Pov}{\mathring{P}^{\vee}}
\newcommand{\Imin}{I_{\mathrm{min}}}
\newcommand{\Scal}{\mathcal{S}}
\newcommand{\X}{\mathcal{X}}
\newcommand{\gh}{\gamma_{h}}
\newcommand{\gv}{\gamma_{v}}
\newcommand{\T}{\mathcal{T}}
\newcommand{\U}{\mathcal{U}}
\newcommand{\UaffA}{U_{q}(A_{1}^{(1)})}
\newcommand{\UaffAA}{U_{q}(A_{2}^{(1)})}
\newcommand{\UaffC}{U_{q}(C_{2}^{(1)})}
\newcommand{\UaffG}{U_{q}(G_{2}^{(1)})}
\newcommand{\A}{\mathcal{A}}
\newcommand{\Hcal}{\mathcal{H}}
\newcommand{\Z}{\mathcal{Z}}

\newcommand{\Tb}{\textbf{T}}
\newcommand{\Xb}{\textbf{X}}
\newcommand{\xbpm}{\mathbf{x}^{\pm}}
\newcommand{\xb}{\mathbf{x}}
\newcommand{\kb}{\mathbf{k}}
\newcommand{\Cb}{\mathbf{C}}
\newcommand{\kk}{\mathbbm{k}}
\newcommand{\s}{\mathfrak{s}}
\newcommand{\xhpm}{\hat{x}^{\pm}}
\newcommand{\xhmp}{\hat{x}^{\mp}}
\newcommand{\xhp}{\hat{x}^{+}}
\newcommand{\xhm}{\hat{x}^{-}}

\newcommand{\kh}{\hat{k}}
\newcommand{\Ch}{\hat{C}}

\numberwithin{equation}{section}

\usepackage{hyperref}
\hypersetup{
    pdftitle={Duncan Laurie - Automorphisms of quantum toroidal algebras from an action of the extended double affine braid group},
    pdfauthor={Duncan Laurie},
    hidelinks,
    pdfhighlight={/N},
    colorlinks=true,
    linkcolor=blue,
    citecolor=blue,
    urlcolor=blue,
    breaklinks=true,
    }
\newcommand\blfootnote[1]{%
  \begingroup
  \renewcommand\thefootnote{}\footnote{#1}%
  \addtocounter{footnote}{-1}%
  \endgroup
}

\usepackage{titling}
\begin{document}

% Setting up theorem styles
\newtheoremstyle{thmstyleone}% name of the style to be used
  {\topsep}% measure of space to leave above the theorem. E.g.: 3pt
  {\topsep}% measure of space to leave below the theorem. E.g.: 3pt
  {\itshape}% name of font to use in the body of the theorem
  {0pt}% measure of space to indent
  {\bfseries}% name of head font
  {}% punctuation between head and body
  {.5em}% space after theorem head; " " = normal interword space
  {\thmname{#1}\hspace{.3em}\thmnumber{#2.}}

\newtheoremstyle{thmstyletwo}
  {\topsep}
  {\topsep}
  {\normalfont}
  {0pt}
  {\itshape}
  {}
  {.5em}
  {\thmname{#1}\hspace{.3em}\thmnumber{#2.}}

\newtheoremstyle{thmstylethree}
  {\topsep}
  {\topsep}
  {\normalfont}
  {0pt}
  {\bfseries}
  {}
  {.5em}
  {\thmname{#1}\hspace{.3em}\thmnumber{#2.}}

% Theorem etc environments
\theoremstyle{thmstyleone}
\newtheorem{thm}{Theorem}[section]
\newtheorem{prop}[thm]{Proposition}
\newtheorem{lem}[thm]{Lemma}
\newtheorem{cor}[thm]{Corollary}

\theoremstyle{thmstyletwo}
\newtheorem{eg}[thm]{Example}
\newtheorem{rmk}[thm]{Remark}

\theoremstyle{thmstylethree}
\newtheorem{defn}[thm]{Definition}

\raggedbottom

\title{Automorphisms of quantum toroidal algebras from an action of the extended double affine braid group}
\author{Duncan Laurie}
\affil{\normalsize{Mathematical Institute, University of Oxford, Andrew Wiles Building,} \\ \normalsize{Woodstock Road, Oxford, OX2 6GG, United Kingdom.}}
\date{}

\maketitle\blfootnote{E-mail: \url{duncan.laurie@maths.ox.ac.uk}}\blfootnote{ORCID: 0009-0006-9331-4835.}\blfootnote{2020 \emph{Mathematics subject classification}: 17B37, 17B67, 20F36, 81R50.}\blfootnote{Key words and phrases: quantum toroidal algebra, quantum affine algebra, quantum affinization, extended double affine braid group, braid group action.}

\vspace{-40pt}

\begin{abstract}
    We first construct an action of the extended double affine braid group $\Bdd$ on the quantum toroidal algebra $\Utor$ in untwisted and twisted types.
    As a crucial step in the proof, we obtain a finite Drinfeld new style presentation for a broad class of quantum affinizations.
    In the simply laced cases, using our action and certain involutions of $\Bdd$ we produce automorphisms and anti-involutions of $\Utor$ which exchange the horizontal and vertical subalgebras.
    Moreover, they switch the central elements $C$ and $k_{0}^{a_{0}}\dots k_{n}^{a_{n}}$ up to inverse.
    This can be viewed as the analogue, for these quantum toroidal algebras, of the duality for double affine braid groups used by Cherednik to realise the difference Fourier transform in his celebrated proof of the Macdonald evaluation conjectures.
    Our work generalises existing results in type $A$ due to Miki which have been instrumental in the study of the structure and representation theory of $U_{q}(\mathfrak{sl}_{n+1,\mathrm{tor}})$.
\end{abstract}

\tableofcontents

\section{Introduction}

Quantum affine algebras were originally introduced by Drinfeld and Jimbo as the quantum groups $\Uaff$ associated to affine Kac-Moody algebras.
Subsequently, Drinfeld \cite{Drinfeld88} provided an alternative realization of $\Uaff$ in the untwisted case as a quantum affinization of the corresponding finite quantum group, as well as a similar realization for twisted types.
Proofs of the equivalence of the two presentations were then published in work by Beck \cite{Beck94}, Jing and Zhang \cites{Jing98,JZ07,JZ10} and Damiani \cites{Damiani12,Damiani15}.
This new presentation, known as the `Drinfeld new realization', has played a crucial role in studying the rich representation theory of quantum affine algebras.
For example, Chari and Pressley \cite{CP95} classified the finite dimensional representations in terms of Drinfeld polynomials, and Frenkel and Jing \cites{FJ88,Jing90} constructed vertex representations. \\

Drinfeld's quantum affinization resembles the formation of untwisted affine Lie algebras by adjoining a derivation to a central extension of the loop algebra of a finite dimensional simple Lie algebra.
This procedure can more generally be applied to any Kac-Moody algebra, and takes affine Kac-Moody algebras to `double affine' or toroidal Lie algebras.
Similarly, the quantum affinization process works for the quantum group of any Kac-Moody algebra.
In particular, from quantum affine algebras we obtain the quantum toroidal algebras $\Utor$, which were first introduced in \cites{GKV95,VV96} for type $A_{n}^{(1)}$ and then for arbitrary symmetric type in \cites{Jing98(2),Nak01}.
Furthermore, Nakajima \cite{Nak01} constructed representations of simply laced quantum affinizations geometrically, using the equivariant K-theory of quiver varieties. For the specific case of simply laced quantum toroidal algebras, see also \cites{Nak02,VV99}. \\

Just as in the quantum affine setting, representations of $\Utor$ are equipped with a level determined by the action of the central elements.
It is known that $\Utor$ contains horizontal and vertical subalgebras $\Uh$ and $\Uv$, each isomorphic to a quantum affine algebra.
In particular, $\Uh$ is the natural copy of $\Uaff$ from which $\Utor$ is formed via quantum affinization, and $\Uv$ is the quantum affinization of the finite quantum group lying inside it.
Then a representation of $\Utor$ is said to have level $(a,b)$ if $\Uv$ acts with level $a$ and $\Uh$ with level $b$. \\

In type $A_{n}^{(1)}$, Varagnolo and Vasserot \cite{VV96} established a Schur-Weyl duality between representations of $\UtorA$ and those of the double affine Hecke algebra.
This duality was then used to construct a level $(0,1)$ action on the $q$-Fock space \cites{STU98,VV98}.
Note that \cite{STU98} also proves the irreducibility of the representation.
Nagao \cite{Nag09b} showed that this is isomorphic to Nakajima's geometric representation in type $A_{n}^{(1)}$ -- torus fixed points on the equivariant K-theory side are identified with certain simultaneous eigenvectors in the $q$-Fock space, defined using non-symmetric Macdonald polynomials. \\

There is also a level $(1,0)$ vertex representation of $\UtorA$ due to Saito \cite{Saito98} (for arbitrary symmetric types see \cite{Jing98(2)}).
Motivated by trying to understand the relationship with the $q$-Fock space representation, Miki \cite{Miki99} constructed an automorphism of $\UtorA$ using a toroidal braid group action.
In particular, his automorphism exchanges the horizontal and vertical subalgebras and swaps their central elements up to inverse.
Miki \cite{Miki00} then used this automorphism to study the representation theory of $\UtorA$, obtaining among other things a classification by Drinfeld polynomials of the irreducibles in a natural class of highest weight representations, and $R$-matrices on tensor products of these modules.
He also clarified the relation between the vertex and $q$-Fock space representations. \\

Surprisingly, relatively little has been written about quantum toroidal algebras outside of type $A_{n}^{(1)}$.
The primary aim of this paper is to generalise the results of \cite{Miki99} to other types.
In particular, we first construct an action of the extended double affine braid group $\Bdd$ on the quantum toroidal algebra $\Utor$ in all untwisted and twisted types other than $A_{1}^{(1)}$ and $A_{2}^{(2)}$.
Then in the simply laced case we use this action to obtain automorphisms and anti-involutions of $\Utor$ which exchange the horizontal and vertical subalgebras.
\\

We expect that -- as in type $A_{n}^{(1)}$ -- this will be helpful for studying the representation theory of $\Utor$, and we plan to explore these directions in future work.
For example, Hernandez \cite{Hernandez05} obtained a Chari-Pressley style classification of the (type $1$) irreducible integrable loop-highest weight modules in terms of Drinfeld polynomials.
Conjugating the Drinfeld topological coproduct of $\Utor$ by our automorphism should produce a tensor product on representations that is well-defined for these modules.
Furthermore, twisting the vertex representation of $\Utor$ from \cite{Jing98} by our automorphism should land within this classification, and moreover relate to Nakajima's geometric representation in all simply laced $ADE$ types.
The author also plans to utilise our action of $\Bdd$ to obtain further analogues of various braid group phenomena inside the quantum algebra setting. \\

This paper is organised as follows.
Section \ref{Preliminaries} establishes our notational conventions surrounding the basic structures used in the theory of affine Kac-Moody algebras.
In Section \ref{Affine Section} we recall the affine situation in more detail, in particular quantum affine algebras; quantum affinization; Jing's isomorphism \cite{Jing98} between the two presentations of untwisted $\Uaff$; and the action of the extended affine braid group due to Lusztig \cite{Lusztig93} and Beck \cite{Beck94}. \\

In Section \ref{The Toroidal Situation} we move to the toroidal setting.
We define the quantum toroidal algebra, proving that it is generated by its horizontal and vertical subalgebras, and present some natural (anti-)automorphisms.
Section \ref{Extended double affine braid group section} introduces the extended double affine braid group $\Bdd$ together with its horizontal and vertical subgroups $\Bh$ and $\Bv$, each of which is isomorphic to an extended affine braid group.
For a broad class of quantum affinizations, we then prove a simplified Drinfeld new style presentation involving only finitely many generators and relations (Proposition~\ref{simpler toroidal presentation}).
In particular, this includes the quantum toroidal algebras $\Utor$ in all types other than $A_{1}^{(1)}$ and $A_{2}^{(2)}$, as well as the untwisted quantum affine algebras $\Uaff$.
This allows us to define automorphisms $\T_{0},\dots,\T_{n}$ of $\Utor$ which restrict to the braid automorphisms of Lusztig \cite{Lusztig93} on both the horizontal and vertical subalgebras (Proposition~\ref{Ti properties}).
We are then able to give an action of $\Bdd$ on $\Utor$ in all types (Theorem~\ref{toroidal braid group action theorem}).
The horizontal and vertical subgroups $\Bh$ and $\Bv$ restrict to the extended affine action of Lusztig and Beck on the horizontal and vertical subalgebras $\Uh$ and $\Uv$ respectively. \\

The braid group $\Bdd$ possesses a natural involution $\te$ which interchanges its horizontal and vertical subgroups.
In Section \ref{quantum toroidal automorphism section}, using the action on $\Utor$ we transfer this over to an anti-involution $\psi$ of the quantum toroidal algebra in all simply laced types (Theorem~\ref{psi theorem}).
Moreover, $\psi$ exchanges the horizontal and vertical subalgebras and acts on central elements by $C \leftrightarrow (k_{0}^{a_{0}}\dots k_{n}^{a_{n}})^{-1}$.
Composing $\psi$ with a standard anti-automorphism $\eta$, we get an automorphism $\Phi$ of $\Utor$ which in type $A_{n}^{(1)}$ recovers that of Miki\footnote{More specifically, Miki considers a quantum toroidal algebra $U_{q,\kappa}(\mathfrak{sl}_{n+1,\mathrm{tor}})$ involving an extra deformation parameter $\kappa$ which is not known to exist in other types.
Our automorphism $\Phi$ in type $A_{n}^{(1)}$ is equal to that of Miki with $\kappa$ set to $1$.}
\cite{Miki99} (Corollary~\ref{Phi corollary}).
We conclude by proving compatibility relations between $\psi$ and $\Phi^{\pm 1}$ and various involutions of $\Bdd$ (Proposition~\ref{compatibilities proposition}).
\\

\textbf{Acknowledgements.}
I would like to thank my supervisor, Kevin McGerty, for many helpful discussions throughout the preparation of this paper — his guidance and encouragement have been invaluable.
This research was financially supported by the Engineering and Physical Sciences Research Council [grant number EP/T517811/1].

\section{Preliminaries} \label{Preliminaries}

For our conventions on affine Kac-Moody algebras, we mostly follow \cite{Kac90}.
We shall consider an affine Kac-Moody algebra $\gaff$ with Cartan matrix $A = (a_{ij})_{i,j\in I}$ and index set $I = \lbrace 0,\dots,n\rbrace$.
It has a Cartan subalgebra $\hat{\mathfrak{h}}$ containing simple coroots $\alpha_{i}^{\vee}$ and fundamental coweights $\lambda_{i}^{\vee}$ for each $i\in I$, which form bases for the coroot and coweight lattices $Q^{\vee}$ and $P^{\vee}$.
The simple roots $\alpha_{i}$ and fundamental weights $\lambda_{i}$ for each $i\in I$ lie in the dual space $\hat{\mathfrak{h}}^{*}$ and span the root and weight lattices $Q$ and $P$.
As $\mathbb{Q}$-vector spaces, $\hat{\mathfrak{h}}$ and $\hat{\mathfrak{h}}^{*}$ have bases $\lbrace\lambda^{\vee}_{0},\alpha^{\vee}_{0},\dots,\alpha^{\vee}_{n}\rbrace$ and $\lbrace\lambda_{0},\alpha_{0},\dots,\alpha_{n}\rbrace$ respectively.
The affine Weyl group $W = \langle s_{i} : i\in I\rangle$ acts on $P^{\vee}$ via $s_{i}(x) = x - \langle\alpha_{i},x\rangle\alpha_{i}^{\vee}$ for each $i\in I$, where $\langle~,~\rangle$ is the natural pairing between $\hat{\mathfrak{h}}$ and $\hat{\mathfrak{h}}^{*}$. \\

Each node $i\in I$ of the affine Dynkin diagram $D(A)$ has a numerical label $a_{i}$, and a dual label $a_{i}^{\vee}$ coming from the diagram with the same vertex numbering and all arrows reversed.
The affine Dynkin diagrams, together with their $a_{i}$ and $a_{i}^{\vee}$ labels, are given in Appendix~\ref{Affine Dynkin Diagrams appendix} -- our choice of vertex numbering matches \cite{Kac90}*{Chapter 4}.
Note that $\delta = \sum_{i\in I} a_{i}\alpha_{i}$ is the standard non-divisible imaginary root in $Q$, and that outside type $A_{2n}^{(2)}$ we have $a_{0} = 1$. \\

The corresponding finite dimensional simple Lie algebra $\g$ has Cartan matrix $(a_{ij})_{i,j\in I_{0}}$ where $I_{0} = \lbrace 1,\dots,n\rbrace$.
It has simple roots $\alpha_{i}$, simple coroots $\alpha_{i}^{\vee}$, fundamental weights $\omega_{i}$, and fundamental coweights $\omega_{i}^{\vee}$ for each $i\in I_{0}$ and we denote its root, coroot, weight and coweight lattices by $\mathring{Q}$, $\Qov$, $\mathring{P}$ and $\Pov$.
By mapping each $\omega_{i}^{\vee} \mapsto a_{0}\lambda_{i}^{\vee} - a_{i}\lambda_{0}^{\vee}$ we can embed $\Pov$ inside $P^{\vee}$ at level $0$, so that $\langle\delta,\omega_{i}^{\vee}\rangle = 0$ for all $i\in I_{0}$.
The image is invariant under the action of the finite Weyl group $W_{0} = \langle s_{i} : i\in I_{0}\rangle$.
Similarly, we can view $\mathring{P}$ inside the affine weight lattice $P$ by sending each
$\omega_{i} \mapsto a_{0}^{\vee}\lambda_{i} - a_{i}^{\vee}\lambda_{0}$.
In order to simplify our notation in later sections we shall moreover define $\omega_{0}^{\vee} = 0$ and $\omega_{0} = 0$. \\

We denote by $\Omega$ the group of outer automorphisms of the affine Dynkin diagram, which is the quotient of the automorphism group of $D(A)$ by the subgroup which fixes the $0$ vertex (and thus restricts to automorphisms of the finite Dynkin diagram).
Elements of $\Omega$ are indexed by $\Imin \subset \lbrace i\in I : a_{i} = a_{0}\rbrace$. In particular, for each $i\in \Imin$ we define $\pi_{i}$ to be the unique outer automorphism sending $0$ to $i$. \\

The affine Cartan matrix $A$ is symmetrized by the diagonal matrix $D = \mathrm{diag}(d_{0},\dots,d_{n})$ where each $d_{i} = a_{i}^{\vee}a_{i}^{-1}$, which is to say that the product $DA$ is symmetric.
The standard non-degenerate symmetric bilinear form $(~,~)$ on $\hat{\mathfrak{h}}^{*}$ is defined by
\begin{align*}
    (\alpha_{i},\alpha_{j}) = d_{i}a_{ij}, \quad
    (\alpha_{i},\lambda_{0}) = d_{0}\delta_{i0}, \quad
    (\lambda_{0},\lambda_{0}) = 0,
\end{align*}
for all $i,j\in I$ and in particular satisfies $(\delta,\alpha_{i}) = 0$.
The corresponding isomorphism $\nu : \hat{\mathfrak{h}} \rightarrow \hat{\mathfrak{h}}^{*}$ maps each $\alpha_{i}^{\vee}\mapsto d_{i}^{-1}\alpha_{i}$ and sends $\lambda_{0}^{\vee}\mapsto d_{0}^{-1}\lambda_{0}$.
Throughout this paper we shall occasionally identify the elements of $\hat{\mathfrak{h}}$ with their images under $\nu$ without mention.
\\

The affine braid group $\B$ has a Coxeter presentation generated by $\lbrace T_{i} : i\in I\rbrace$ subject to braid relations $T_{i}T_{j}T_{i}\ldots = T_{j}T_{i}T_{j}\ldots$ with $a_{ij}a_{ji} + 2$ factors on each side (except in types $A_{1}^{(1)}$ and $A_{2}^{(2)}$ where $T_{0}$ and $T_{1}$ satisfy no relation).
This is clearly independent of the orientation of arrows in the underlying Dynkin diagram, and so any affine braid group is isomorphic to one of untwisted type.
\\

In all untwisted and $A_{2n}^{(2)}$ types, let $M = \Qov$ and $A_{i}^{\vee} = \alpha_{i}$ for each $i\in I$.
Conversely, in the remaining twisted types we define $M = \mathring{Q}$ and all $A_{i}^{\vee} = \alpha_{i}^{\vee}$.
Then in each case, the Bernstein presentation of $\B$ is generated by the finite braid group $\B_{0} = \langle T_{i} : i\in I_{0}\rangle$ and the lattice $\lbrace X_{\beta} : \beta \in M \rbrace$, with
\begin{itemize}
    \item $T_{i}X_{\beta} = X_{\beta}T_{i}$ if $(\beta,A_{i}^{\vee}) = 0$,
    \item $T_{i}^{-1}X_{\beta}T_{i}^{-1} = X_{s_{i}(\beta)}$ if $(\beta,A_{i}^{\vee}) = 1$.
\end{itemize}

When $M = \Qov$ the correspondence between the two presentations is given by $T_{0} = X_{\theta^{\vee}} T_{s_{\theta}}^{-1}$ where $\theta = \sum_{i\in I_{0}} a_{i}\alpha_{i}$ is the highest root of $\g$ and $\theta^{\vee} = \nu^{-1}(a_{0}^{-1}\theta)$.
Otherwise, $\theta$ is the short dominant root in $M = \mathring{Q}$ and we instead have $T_{0} = X_{\theta} T_{s_{\theta}}^{-1}$.
See \cite{IS20}*{Chapter 3} for more details, noting that the Bernstein presentation there is obtained from ours by applying the automorphism of $\B$ which inverts $T_{1},\dots,T_{n}$ and fixes each $X_{\beta}$.
\\

Throughout this paper we shall work over the field $\kk = \mathbb{Q}(q^{\min\lbrace d_{i} \rbrace})$.
Setting $q_{i} = q^{d_{i}}$ for all $i\in I$, the $q_{i}$-integers, $q_{i}$-factorials and $q_{i}$-binomial coefficients are defined as
\begin{align*}
    [s]_{i} = \frac{q_{i}^{s}-q_{i}^{-s}}{q_{i}-q_{i}^{-1}},
    \quad
    [s]_{i}! = \prod_{\ell=1}^{s} [\ell]_{i},
    \quad
    \begin{bmatrix}{s}\\ {r}\end{bmatrix}_{i} = \frac{[s]_{i}!}{[s-r]_{i}!\,[r]_{i}!}
\end{align*}
respectively for all non-negative integers $s\geq r$.
We then let $(\xipm)^{(s)} = (\xipm)^{s}/[s]_{i}!$ and $(\xpm_{i,m})^{(s)} = (\xpm_{i,m})^{s}/[s]_{i}!$ for elements $\xipm$ and $\xpm_{i,m}$ of certain quantum algebras defined in later sections.
Following Jing \cite{Jing98} we define twisted commutators $[b_{1},\dots,b_{s}]_{u_{1}\cdots u_{s-1}}$ inductively from $[b_{1},b_{2}]_{u} = b_{1}b_{2} - u b_{2}b_{1}$ and
\begin{align*}
    [b_{1},\dots,b_{s}]_{u_{1}\cdots u_{s-1}} = [b_{1},[b_{2},\dots,b_{s}]_{u_{1}\cdots u_{s-2}}]_{u_{s-1}}.
\end{align*}

Outside of type $A_{2n}^{(1)}$ we can fix a length function $o:I\rightarrow\lbrace\pm 1\rbrace$ satisfying $o(i) = -o(j)$ whenever $a_{ij}<0$.
We shall write $o_{i,j}$ as shorthand for $o(i)/o(j)$.
However, in type $A_{2n}^{(1)}$ this is not possible since the affine Dynkin diagram contains an odd length cycle.
For our purposes, there are two approximations to a length function to consider in this case: $o(i) = (-1)^{i}$ and $-o(i) = (-1)^{i+1}$.
Furthermore, we define $o_{i,j} = (-1)^{\overline{j-i}}$ for all $i,j\in I$, where $\overline{j-i}$ is the anti-clockwise distance $i\rightarrow j$ in the affine Dynkin diagram.

\section{The affine situation} \label{Affine Section}

In this section we introduce the quantum affine algebras $\Uaff$, and outline in the untwisted case their alternative Drinfeld new presentation as the quantum affinizations of finite quantum groups.
We then present the automorphisms of $\Uaff$ which form the action of the extended affine braid group $\Bd$ due to Lusztig \cite{Lusztig93} and Beck \cite{Beck94}.

\subsection{Quantum affine algebras} \label{quantum affine algebra subsection}

For an arbitrary symmetrizable Kac-Moody algebra $\s$ with generalized Cartan matrix $(a_{ij})_{i,j\in I}$, the corresponding quantum group is given in terms of certain Chevalley style generators as follows.

\begin{defn} \label{quantum group definition}
The quantum group $U_{q}(\s)$ is the unital associative $\kk$-algebra generated by elements $x_{i}^{\pm}$ and $t_{i}^{\pm 1}$ for each $i\in I$, subject to the following relations:
\begin{itemize}
    \item $\displaystyle t_{i}^{\pm 1} t_{i}^{\mp 1} = 1$,
    \item $\displaystyle [t_{i},t_{j}] = 0$,
    \item $\displaystyle t_{i} \xjpm t_{i}^{-1} = q_{i}^{\pm a_{ij}} \xjpm$,
    \item $\displaystyle [\xip,\xjm] = \frac{\delta_{ij}}{q_{i}-q_{i}^{-1}} (t_{i} - t_{i}^{-1})$,
    \item $\displaystyle \sum_{s=0}^{1-a_{ij}} (-1)^{s} (\xipm)^{(s)} \xjpm (\xipm)^{(1-a_{ij}-s)} = 0$ whenever $i\not= j$.
\end{itemize}
\end{defn}

This is called the Drinfeld-Jimbo presentation of the quantum group.
In particular, for any affine Kac-Moody algebra $\gaff$ we have an associated quantum affine algebra $\Uaff$.
Depending on the context, some authors include an extra degree operator in their definition for $\Uaff$.
However, we shall not do so in this paper.
\\

In the untwisted case $\Uaff$ has an alternative \textit{Drinfeld new presentation}, first stated by Drinfeld \cite{Drinfeld88}, as the quantum affinization of the finite quantum group $\Uq$.
One may view this as a deformation quantization of the one-dimensional central extension of the loop algebra $\g[t,t^{-1}]$.
Loosely speaking, the $\xp_{i,m},\xm_{i,m},h_{i,r},k_{i}$ generators below correspond to the elements $e_{i}t^{m},f_{i}t^{m},h_{i}t^{r},h_{i}$ respectively inside $\g[t,t^{-1}]$, and $C$ is identified with the central extension.

\begin{defn} \label{quantum affinization definition}
    The quantum affinization of $U_{q}(\s)$ is the unital associative $\kk$-algebra $\widehat{U_{q}(\s)}$ with generators $\xpm_{i,m}$, $h_{i,r}$, $k_{i}^{\pm 1}$, $C^{\pm 1}$ ($i\in I$, $m\in\mathbb{Z}$, $r\in\mathbb{Z}^{*}$) and relations
\begin{itemize}
    \item $C^{\pm 1}$ central,
    \item $\displaystyle C^{\pm 1}C^{\mp 1} = k_{i}^{\pm 1} k_{i}^{\mp 1} = 1$,
    \item $\displaystyle [k_{i},k_{j}] = [k_{i},h_{j,r}] = 0$,
    \item $\displaystyle [h_{i,r},h_{j,s}] = \delta_{r+s,0} \frac{[ra_{ij}]_{i}}{r} \frac{C^{r}-C^{-r}}{q_{j}-q_{j}^{-1}}$,
    \item $\displaystyle k_{i} \xpm_{j,m} k_{i}^{-1} = q_{i}^{\pm a_{ij}} \xpm_{j,m}$,
    \item $\displaystyle [h_{i,r},\xpm_{j,m}] = \pm \frac{[ra_{ij}]_{i}}{r} C^{\frac{r \mp |r|}{2}} \xpm_{j,r+m}$,
    \item $\displaystyle [\xp_{i,m},\xm_{j,l}] = \frac{\delta_{ij}}{q_{i}-q_{i}^{-1}} (C^{-l}\phi^{+}_{i,m+l} - C^{-m}\phi^{-}_{i,m+l})$,
    \item $[\xpm_{i,m+1},\xpm_{j,l}]_{q_{i}^{\pm a_{ij}}} + [\xpm_{j,l+1},\xpm_{i,m}]_{q_{i}^{\pm a_{ij}}} = 0$,
\end{itemize}
and whenever $i\not= j$, for any integers $m$ and $m_{1},\dots,m_{a'}$ where $a' = 1 - a_{ij}$,
\begin{itemize}
    \item $\displaystyle
    \sum_{\pi\in S_{a'}}
    \sum_{s=0}^{a'} (-1)^{s}
    {\begin{bmatrix}a'\\s\end{bmatrix}}_{i}
    \xpm_{i,m_{\pi(1)}}\dots\xpm_{i,m_{\pi(s)}}
    \xpm_{j,m}
    \xpm_{i,m_{\pi(s+1)}}\dots\xpm_{i,m_{\pi(a')}}
    = 0$.
\end{itemize}
Here, the $\phi^{\pm}_{i,\pm s}$ are given by the formula
$$ \sum_{s\geq 0} \phi^{\pm}_{i,\pm s} z^{\pm s} =
k_{i}^{\pm 1} \exp{\left( \pm (q_{i}-q_{i}^{-1})\sum_{s'>0}h_{i,\pm s'} z^{\pm s'} \right)}
$$
when $s\geq 0$, and are zero otherwise.
\end{defn}

The relationship between these two presentations of $\Uaff$ in the untwisted case was first studied by Beck \cite{Beck94}, who used an action of the extended affine braid group to construct a morphism from the Drinfeld new realization to the Drinfeld-Jimbo realization.
Jing \cite{Jing98} then defined an inverse morphism using $q$-commutators, while Damiani proved the surjectivity \cite{Damiani12} and injectivity \cite{Damiani15} of Beck's map.

\begin{rmk}
    The definition of quantum affinization varies slightly between sources.
    We use the one found for example in \cites{Damiani12,Miki99} since it is more precise regarding the isomorphism between the two presentations of $\Uaff$.
    The definition found in other works such as \cites{Beck94,Jing98,Hernandez09} can then be obtained by adjoining $C^{\pm 1/2}$ and scaling each $\xpm_{i,m}$ generator by $C^{m/2}$.
\end{rmk}

Let us now present Jing's isomorphism.
For each $i_{1}\in I_{0}$ there exist sequences $\underline{i} = (i_{1},i_{2},\dots,i_{h-1})$ in $I_{0}$ and $\underline{\epsilon} = (\epsilon_{1},\dots,\epsilon_{h-2})$ in $\mathbb{Q}_{\leq 0}$ such that
\begin{align} \label{Jing's isomorphism condition}
    (\alpha_{i_{1}}+\dots +\alpha_{i_{s}}\vert\alpha_{i_{s+1}}) = \epsilon_{s} \mathrm{~for~} s=1,\dots,h-2,
\end{align}
where $h = \sum_{i\in I}a_{i}$ is the Coxeter number of $\gaff$.
Then for any such sequences, the following extends to a $\kk$-algebra isomorphism from the Drinfeld-Jimbo realization of $\Uaff$ to the Drinfeld new realization:
\begin{itemize}
    \item $\xipm \mapsto \xpm_{i,0}$ and $t_{i} \mapsto k_{i}$ for each $i\in I_{0}$,
    \item $\xp_{0} \mapsto \left[\xm_{i_{h-1},0},\dots,\xm_{i_{2},0},\xm_{i_{1},1}\right]_{q^{\epsilon_{1}}\dots q^{\epsilon_{h-2}}} C k_{\theta}^{-1}$,
    \item $\xm_{0} \mapsto a(-q)^{-\epsilon} C^{-1} k_{\theta} \left[\xp_{i_{h-1},0},\dots,\xp_{i_{2},0},\xp_{i_{1},-1}\right]_{q^{\epsilon_{1}}\dots q^{\epsilon_{h-2}}}$,
    \item $t_{0} \mapsto C k_{\theta}^{-1}$,
\end{itemize}
where $k_{\theta} = k_{1}^{a_{1}}\dots k_{n}^{a_{n}}$, $\epsilon = \epsilon_{1}+\dots+\epsilon_{h-2}$, and $a$ is a constant depending on type (in particular $a=1$ when $\gaff$ is simply laced). Example sequences in all types can be found in \cite{Jing98}*{Table 2.1}.

\begin{rmk}
    It is clear in both presentations that $\Uaff$ contains a natural copy of the finite quantum group $\Uq$ -- it is the subalgebra generated by $\lbrace \xipm, t_{i}^{\pm 1} : i\in I_{0}\rbrace$ in the Drinfeld-Jimbo, and by $\lbrace \xpm_{i,0}, k_{i}^{\pm 1} : i\in I_{0}\rbrace$ in the Drinfeld new.
\end{rmk}

So we see that $\Uaff$ can formed from $\g$ either as the quantum group of the affine Kac-Moody algebra $\gaff$, or by performing quantum affinization to $\Uq$.
This is precisely the commutativity of the following diagram, taken from \cite{Hernandez09}.
\[\begin{tikzcd}
	\g & & & & \gaff \\
	\Uq & & & & \Uaff
	\arrow["\mathrm{Quantization}"', from=1-1, to=2-1]
	\arrow["\mathrm{Quantum~Affinization}"', from=2-1, to=2-5]
	\arrow["\mathrm{Affinization}", from=1-1, to=1-5]
	\arrow["\mathrm{Quantization}", from=1-5, to=2-5]
\end{tikzcd}\]
Since quantum affinization is defined for the quantum group of any Kac-Moody algebra, we can apply it to $\Uaff$ to obtain a sort of `double affine' quantum group.
As we will see in Section \ref{The Toroidal Situation}, this is precisely the quantum toroidal algebra $\Utor$.
It should be noted that $\Utor$ is not the quantum group of any Kac-Moody algebra, and so cannot by further affinized in this way.

\begin{rmk}
    In fact, in twisted types there is also a Drinfeld new realization of $\Uaff$.
    A morphism from the Drinfeld-Jimbo presentation was defined by Jing and Zhang \cites{JZ07,JZ10}, while the proof that it is an isomorphism was again completed by Damiani in \cites{Damiani12,Damiani15}.
    However, we do not include these cases here since they are not required for our purposes.
\end{rmk}

\subsection{Automorphisms and anti-automorphisms} \label{Automorphisms of Uaff section}

Next we shall present various automorphisms and anti-automorphisms of quantum groups and their quantum affinizations.
These are required to describe Lusztig and Beck's action of the extended affine braid group on $\Uaff$, and also for our work towards a corresponding toroidal result in Section \ref{The Toroidal Situation}. \\

First consider a quantum group $U_{q}(\s)$ coming from a generalized Cartan matrix $(a_{ij})_{i,j\in I}$ as explained in Section \ref{quantum affine algebra subsection}.
\begin{itemize}
    \item For each $i\in I$ there is an automorphism $\Tb_{i}$ defined by $\Tb_{i}(t_{j}) = t_{j} t_{i}^{-a_{ij}}$ and
    \begin{align*}
        &\Tb_{i}(\xip) = -\xim t_{i}, ~ \Tb_{i}(\xjp) = \sum_{s=0}^{-a_{ij}} (-1)^{s} q_{i}^{-s} (\xip)^{(-a_{ij}-s)} \xjp (\xip)^{(s)} \mathrm{~if~} i\not= j, \\
        & \Tb_{i}(\xim) = -t_{i}^{-1}\xip, ~ \Tb_{i}(\xjm) = \sum_{s=0}^{-a_{ij}} (-1)^{s} q_{i}^{s} (\xim)^{(s)} \xjm (\xim)^{(-a_{ij}-s)} \mathrm{~if~} i\not= j.
    \end{align*}
    Its inverse $\Tb_{i}^{-1}$ is given by $\Tb_{i}^{-1}(t_{j}) = t_{j} t_{i}^{-a_{ij}}$ and
    \begin{align*}
        &\Tb_{i}^{-1}(\xip) = -t_{i}^{-1}\xim, ~ \Tb_{i}^{-1}(\xjp) = \sum_{s=0}^{-a_{ij}} (-1)^{s} q_{i}^{-s} (\xip)^{(s)} \xjp (\xip)^{(-a_{ij}-s)} \mathrm{~if~} i\not= j, \\
        &\Tb_{i}^{-1}(\xim) = -\xip t_{i}, ~ \Tb_{i}^{-1}(\xjm) = \sum_{s=0}^{-a_{ij}} (-1)^{s} q_{i}^{s} (\xim)^{(-a_{ij}-s)} \xjm (\xim)^{(s)} \mathrm{~if~} i\not= j.
    \end{align*}
    \item Every automorphism $\pi$ of the associated Dynkin diagram gives rise to an automorphism $S_{\pi}$ of $U_{q}(\s)$ which permutes the generators accordingly:
    \begin{align*}
        S_{\pi}(\xjpm) = \xpm_{\pi(j)}, \quad S_{\pi}(t_{j}) = t_{\pi(j)}.
    \end{align*}
    \item There is an anti-involution $\sigma$ such that
    $\sigma(\xjpm) = \xjpm$ and $\sigma(t_{j}) = t_{j}^{-1}$.
    A quick check verifies that $\Tb_{i}^{-1} = \sigma \Tb_{i} \sigma$ for all $i\in I$.
\end{itemize}
Throughout this paper we shall use without comment that $\Tb_{i} \Tb_{j} (\xipm) = \xjpm$ and $\Tb_{i}^{-1} \Tb_{j}^{-1} (\xipm) = \xjpm$ whenever $a_{ij} = a_{ji} = -1$.

\begin{rmk}
    The automorphisms $\Tb_{i}$ and $\Tb_{i}^{-1}$ were first introduced in the general case by Lusztig \cite{Lusztig93}*{Chapter 37}, who denoted them by $T_{i,1}''$ and $T_{i,-1}'$ respectively.
\end{rmk}

Now consider the quantum affinization $\widehat{U_{q}(\s)}$ introduced in Definition \ref{quantum affinization definition}.
\begin{itemize}
    \item For each $i\in I$ there is an automorphism $\X_{i}$ given by
    \begin{align*}
        &\X_{i}(\xpm_{j,m}) = \upsilon(j)^{\delta_{ij}} \xpm_{j,m\mp\delta_{ij}}, \quad \X_{i}(h_{j,r}) = h_{j,r}, \\
        &\X_{i}(k_{j}) = C^{-\delta_{ij}} k_{j}, \quad \X_{i}(C) = C,
    \end{align*}
    where $\upsilon$ is any $\lbrace \pm 1 \rbrace$-valued function on $I$, for example a length function.
    \item There is also an anti-involution $\eta$ with
    \begin{align*}
        &\eta(\xpm_{i,m}) = \xpm_{i,-m}, \quad \eta(h_{i,r}) = -C^{r} h_{i,-r}, \quad \eta(k_{i}) = k_{i}^{-1}, \quad \eta(C) = C.
    \end{align*}
\end{itemize}
For untwisted $\Uaff$ in particular, considered with respect to the Drinfeld new realization and letting $\upsilon$ be the length function $o$ on $I_{0}$, we shall denote these by $\Xb_{i}$ and $\eta'$ respectively.
Note that in this case we have $\Tb_{i}^{-1} = \eta' \Tb_{i} \eta'$ for all $i\in I_{0}$. \\

Recall that the quantum toroidal algebra $\Utor$ will be formed as the quantum affinization of $\Uaff$, and therefore has a Drinfeld new style presentation.
In Section \ref{Toroidal Braid Group Action Section} we wish to define automorphisms $\T_{i}$ of $\Utor$ for each $i\in I$ which are comparable to the automorphisms $\Tb_{i}$ of $\Uaff$.
In particular, this requires us to write the actions of $\T_{i}$ on certain Drinfeld new style generators. \\

To this end, in the case of untwisted $\Uaff$ let us derive formulae for some of the $\Tb_{i}(\xpm_{j,m})$ when $i,j\in I_{0}$.
It is clear from the definitions that $\Tb_{i}$ commutes with $\Xb_{j}$ whenever $j\not= i$ and therefore
\begin{align*}
    &\Tb_{i}(\xp_{j,m}) = o(j)^{m}\Xb_{j}^{-m}\Tb_{i}(\xp_{j,0}) = \sum_{s=0}^{-a_{ij}} (-1)^{s} q_{i}^{-s} (\xp_{i,0})^{(-a_{ij}-s)} \xp_{j,m} (\xp_{i,0})^{(s)}, \\
    &\Tb_{i}(\xm_{j,m}) = o(j)^{m}\Xb_{j}^{m}\Tb_{i}(\xm_{j,0}) = \sum_{s=0}^{-a_{ij}} (-1)^{s} q_{i}^{s} (\xm_{i,0})^{(s)} \xm_{j,m} (\xm_{i,0})^{(-a_{ij}-s)},
\end{align*}
for all $m\in\mathbb{Z}$.
When $i=j$ the $\Tb_{i}(\xpm_{j,m})$ are calculated recursively on $|m|$ and expressions quickly become complicated. However, thanks to a simplified presentation of $\Utor$ coming from Proposition~\ref{simpler toroidal presentation}, we shall only require the case $m = 0,\mp 1$.
Let $U_{i}$ be the subalgebra of $\Uaff$ generated by $\lbrace \xpm_{i,m}, h_{i,r}, k_{i}^{\pm 1}, C^{\pm 1} : m\in\mathbb{Z}, r\in\mathbb{Z}^{*} \rbrace$, and $h_{i} : U_{q}(A_{1}^{(1)}) \xrightarrow{\sim} U_{i}$ be the morphism sending
\begin{align*}
    & q \mapsto q_{i}, \quad k_{1} \mapsto k_{i}, \quad k_{0} \mapsto C k_{i}^{-1}, \quad \xpm_{1} \mapsto \xpm_{i,0}, \\
    & \xp_{0} \mapsto - o(i) C k_{i}^{-1} \xm_{i,1}, \quad
    \xm_{0} \mapsto - o(i) \xp_{i,-1} k_{i} C^{-1}.
\end{align*}
Then by Corollary 3.8 of \cite{Beck94} we have $\Tb_{i}\circ h_{i} = h_{i}\circ \Tb_{1}$ and hence
\begin{align*}
    & \Tb_{i}(\xp_{i,-1}) = h_{i}\circ \Tb_{1}(- o(i) \xm_{0} k_{0})
    = \sum_{s=0}^{2} (-1)^{s}q_{i}^{s}(\xm_{i,0})^{(s)}\xp_{i,-1}(\xm_{i,0})^{(2-s)} k_{i}, \\
    & \Tb_{i}(\xm_{i,1}) = h_{i}\circ \Tb_{1}(- o(i) k_{0}^{-1} \xp_{0})
    = k_{i}^{-1} \sum_{s=0}^{2} (-1)^{s}q_{i}^{-s}(\xp_{i,0})^{(2-s)}\xm_{i,1}(\xp_{i,0})^{(s)}.
\end{align*}

\subsection{Extended affine braid groups} \label{Extended affine braid group subsection}
Here we introduce the extended affine braid group $\Bd$ and present its action on $\Uaff$ due to Lusztig \cite{Lusztig93} and Beck \cite{Beck94}.
For a more complete introduction to extended affine braid groups, the interested reader may wish to consult \cite{Mac03}*{Chapters 2-3} and \cite{IS20}*{Chapter 9}. \\

Recall from Section \ref{Preliminaries} the Coxeter and Bernstein presentations of the affine braid group $\B$.
By replacing the lattice $M$ in the latter with a larger lattice $N$, defined to be $\Pov$ in all untwisted and $A_{2n}^{(2)}$ types and $\mathring{P}$ otherwise, we obtain a Bernstein presentation for the extended affine braid group.

\begin{defn} \label{extended affine braid group Bernstein}
    The extended affine braid group $\Bd$ is generated by the finite braid group $\B_{0} = \langle T_{i} : i\in I_{0}\rangle$ and the lattice $\lbrace X_{\beta} : \beta \in N \rbrace$, subject to
    \begin{itemize}
        \item $T_{i}X_{\beta} = X_{\beta}T_{i}$ if $(\beta,A_{i}^{\vee}) = 0, \hfill \refstepcounter{equation}(\theequation)\label{first extended affine Bernstein}$
        \item $T_{i}^{-1}X_{\beta}T_{i}^{-1} = X_{s_{i}(\beta)}$ if $(\beta,A_{i}^{\vee}) = 1. \hfill \refstepcounter{equation}(\theequation)\label{second extended affine Bernstein}$
    \end{itemize}
\end{defn}

There is also a Coxeter style presentation of $\Bd$.
It is clear that $\B_{0}$ and $\lbrace X_{\beta} : \beta \in M \rbrace$ generate a normal subgroup of $\Bd$ isomorphic to $\B$, and therefore $\Bd \cong (\Bd/\B) \ltimes \B$.
When $N = \Pov$ set $\beta_{\theta} = \theta^{\vee}$ and $\beta_{i} = \omega_{i}^{\vee}$ for each $i\in I$, and when $N = \mathring{P}$ set $\beta_{\theta} = \theta$ and each $\beta_{i} = \omega_{i}$.
Let $v_{i} = w_{0}w_{0i}$ where $w_{0}$ is the longest element\footnote{For a nice explanation of how to find a reduced expression for any $w_{0}$ (and thus $w_{0i}$) by $2$-colouring the Dynkin diagram, see Allen Knutson's answer at
\url{https://mathoverflow.net/questions/54926/longest-element-of-weyl-groups}
% Wayback machine link: /web/20230331141338/https://mathoverflow.net/questions/54926/longest-element-of-weyl-groups
(last accessed 31$^{\mathrm{st}}$ March 2023).}
of $W_{0}$ and $w_{0i}$ is the longest element of the isotropy subgroup $\langle s_{j} : j\not= i\rangle$ of $\beta_{i}$.
It was shown in \cite{Mac03}*{Chapter 2} that $\Bd/\B = \lbrace U_{i} = X_{\beta_{i}} T_{v_{i}}^{-1} : i\in \Imin \rbrace$, and further that $\Bd/\B$ acts on $\B$ by outer automorphisms of the affine Dynkin diagram.
More specifically, $U_{i}T_{j}U_{i}^{-1} = T_{\pi_{i}(j)}$ for all $i\in \Imin$ and $j\in I$ and so we have the following.

\begin{prop} \label{extended affine braid group isomorphism}
    The extended affine braid group $\Bd$ is isomorphic to the semidirect product $\Omega \ltimes \B$.
\end{prop}

The correspondence between the Coxeter and Bernstein presentations of $\Bd$ is given by $T_{0} = X_{\beta_{\theta}} T_{s_{\theta}}^{-1}$ and $\pi_{i} = X_{\beta_{i}} T_{v_{i}}^{-1}$ for each $i\in \Imin$.

\begin{rmk} \label{alternative Bernstein presentation}
    There is an automorphism of $\Bd$ which inverts $T_{0},\dots,T_{n}$ and fixes each element of $\Omega$.
    Letting $Y_{\beta}$ be the image of $X_{\beta}$ for all $\beta \in N$, we obtain an \textit{alternative Bernstein presentation} of $\Bd$ matching that of \cite{IS20}*{Proposition 9.1}.
    In particular, for each $i\in I_{0}$ and $\beta \in N$ we have the relations
    \begin{itemize}
        \item $T_{i}Y_{\beta} = Y_{\beta}T_{i}$ if $(\beta,A_{i}^{\vee}) = 0, \hfill \refstepcounter{equation}(\theequation)\label{first alternative extended affine Bernstein}$
        \item $T_{i} Y_{\beta}T_{i} = Y_{s_{i}(\beta)}$ if $(\beta,A_{i}^{\vee}) = 1. \hfill \refstepcounter{equation}(\theequation)\label{second alternative extended affine Bernstein}$
    \end{itemize}
    It immediately follows that the Coxeter presentation relates to this alternative Bernstein presentation via $T_{0} = T_{s_{\theta}}^{-1}Y_{-\beta_{\theta}}$ and $\pi_{i} = Y_{\beta_{i}} T_{v_{i}^{-1}}$ for each $i\in \Imin$.
\end{rmk}

\begin{eg}
We fix natural representatives for $\pi_{i}$ in all affine types where there exists a non-trivial automorphism of the corresponding finite Dynkin diagram.
\begin{itemize}
    \item In type $A_{n}^{(1)}$ we have $\Omega \cong \mathbb{Z}_{n+1}$ by identifying $\pi_{i} = ( j \mapsto j+i~\mathrm{mod}~n+1)$ with $i\in\mathbb{Z}_{n+1}$ for each $i\in I$.
    \item In type $D_{2n}^{(1)}$ we have $\Omega \cong \mathbb{Z}_{2}\times\mathbb{Z}_{2}$ with non-trivial elements given by
    \begin{align*}
        \pi_{1} &= (0\leftrightarrow 1, n-1\leftrightarrow n), \\
        \pi_{n-1} &= (0\leftrightarrow n-1, 1\leftrightarrow n), \\
        \pi_{n} &= (0\leftrightarrow n, 1\leftrightarrow n-1).
    \end{align*}
    \item In type $D_{2n+1}^{(1)}$ we instead have $\Omega \cong \mathbb{Z}_{4}$ with
    \begin{align*}
        \pi_{1} &= (0\leftrightarrow 1, n-1\leftrightarrow n), \\
        \pi_{n-1} &= (0\mapsto n-1\mapsto 1\mapsto n\mapsto 0), \\
        \pi_{n} &= (0\mapsto n\mapsto 1\mapsto n-1\mapsto 0).
    \end{align*}
    \item In type $E_{6}^{(1)}$ we have $\Omega \cong \mathbb{Z}_{3}$ and non-trivial elements
    \begin{align*}
        \pi_{1} = (0\mapsto 1\mapsto 5\mapsto 0), \\
        \pi_{5} = (0\mapsto 5\mapsto 1\mapsto 0).
    \end{align*}
\end{itemize}
\end{eg}

We are now ready to state Lusztig and Beck's affine action.

\begin{thm} \label{Beck's affine result}
    The extended affine braid group $\Bd$ acts on the quantum affine algebra $\Uaff$ via $T_{i} \rightarrow \Tb_{i}$ for each $i\in I$ and $\pi \rightarrow S_{\pi}$ for each $\pi\in\Omega$.
    Furthermore, in all untwisted types it follows that $X_{\omega_{i}^{\vee}}$ acts by the automorphism $\Xb_{i}$ for each $i\in I_{0}$.
\end{thm}

\section{The toroidal situation} \label{The Toroidal Situation}

We now move to the toroidal setting, where we are able to obtain an action of the extended double affine braid group $\Bdd$ on the quantum toroidal algebra $\Utor$ in all untwisted and twisted types other than $A_{1}^{(1)}$ and $A_{2}^{(2)}$.
The construction of certain $\T_{i}$ automorphisms involved in this action requires a simplified Drinfeld new style presentation of $\Utor$ given in terms of finitely many generators and relations.

\subsection{Quantum toroidal algebras} \label{quantum toroidal algebras subsection}

As mentioned in Section \ref{Affine Section}, the quantum toroidal algebra of type $X_{n}^{(r)}$ is defined to be the quantum affinization of the quantum affine algebra $\Uaff$ of type $X_{n}^{(r)}$.

\begin{defn} \label{quantum toroidal algebra definition}
    The quantum toroidal algebra $\Utor$ is the unital associative $\kk$-algebra with generators $\xpm_{i,m}$, $h_{i,r}$, $k_{i}^{\pm 1}$, $C^{\pm 1}$ ($i\in I$, $m\in\mathbb{Z}$, $r\in\mathbb{Z}^{*}$), subject to the following relations:
\begin{itemize}
    \item $C^{\pm 1}$ central,
    \item $\displaystyle C^{\pm 1}C^{\mp 1} = k_{i}^{\pm 1} k_{i}^{\mp 1} = 1$,
    \item $\displaystyle [k_{i},k_{j}] = [k_{i},h_{j,r}] = 0$,
    \item $\displaystyle [h_{i,r},h_{j,s}] = \delta_{r+s,0} \frac{[ra_{ij}]_{i}}{r} \frac{C^{r}-C^{-r}}{q_{j}-q_{j}^{-1}}$,
    \item $\displaystyle k_{i} \xpm_{j,m} k_{i}^{-1} = q_{i}^{\pm a_{ij}} \xpm_{j,m}$,
    \item $\displaystyle [h_{i,r},\xpm_{j,m}] = \pm \frac{[ra_{ij}]_{i}}{r} C^{\frac{r \mp |r|}{2}} \xpm_{j,r+m}$,
    \item $\displaystyle [\xp_{i,m},\xm_{j,l}] = \frac{\delta_{ij}}{q_{i}-q_{i}^{-1}} (C^{-l}\phi^{+}_{i,m+l} - C^{-m}\phi^{-}_{i,m+l})$,
    \item $[\xpm_{i,m+1},\xpm_{j,l}]_{q_{i}^{\pm a_{ij}}} + [\xpm_{j,l+1},\xpm_{i,m}]_{q_{i}^{\pm a_{ij}}} = 0$,
\end{itemize}
and whenever $i\not= j$, for any integers $m$ and $m_{1},\dots,m_{a'}$ where $a' = 1 - a_{ij}$,
\begin{itemize}
    \item $\displaystyle
    \sum_{\pi\in S_{a'}}
    \sum_{s=0}^{a'} (-1)^{s}
    {\begin{bmatrix}a'\\s\end{bmatrix}}_{i}
    \xpm_{i,m_{\pi(1)}}\dots\xpm_{i,m_{\pi(s)}}
    \xpm_{j,m}
    \xpm_{i,m_{\pi(s+1)}}\dots\xpm_{i,m_{\pi(a')}}
    = 0$.
\end{itemize}
Here, the $\phi^{\pm}_{i,\pm s}$ are given by the formula
$$ \sum_{s\geq 0} \phi^{\pm}_{i,\pm s} z^{\pm s} =
k_{i}^{\pm 1} \exp{\left( \pm (q_{i}-q_{i}^{-1})\sum_{s'>0}h_{i,\pm s'} z^{\pm s'} \right)}
$$
when $s\geq 0$, and are zero otherwise.
\end{defn}

\begin{rmk}
    In type $A_{n}^{(1)}$ there is a two-parameter deformation $U_{q,\kappa}(\mathfrak{sl}_{n+1,\mathrm{tor}})$ where some of the relations are modified to involve additional central generators $\kappa^{\pm 1}$.
    Indeed this is the algebra considered by Miki \cite{Miki99}, and specialises to the above at $\kappa = 1$.
    However, such a deformation is not known to exist in other types and thus will not be treated in this paper.
\end{rmk}

So we see that the quantum toroidal algebra $\Utor$ of type $X_{n}^{(r)}$ can be obtained from the corresponding finite quantum group $\Uq$ by affinizing twice on the quantum level.
In fact, $\Utor$ contains two natural quantum affine subalgebras.
There is a horizontal subalgebra $\Uh$ of type $X_{n}^{(r)}$ defined as the image of the homomorphism $h : U_{q}(X_{n}^{(r)}) \rightarrow \Utor$ sending
\begin{align*}
    \xipm \mapsto \xpm_{i,0}, \quad t_{i} \mapsto k_{i},
\end{align*}
for all $i\in I$.
Additionally, there is a vertical subalgebra $\Uv$ of untwisted type $Z_{n}^{(1)}$, where $Z_{n}$ is the finite Cartan type of the simple Lie algebra $\g$.
It is the image of the homomorphism $v : U_{q}(Z_{n}^{(1)}) \rightarrow \Utor$ given by
\begin{align*}
    \xpm_{i,m} \mapsto \xpm_{i,m}, \quad h_{i,r} \mapsto h_{i,r}, \quad k_{i} \mapsto k_{i}, \quad C \mapsto C,
\end{align*}
for all $i\in I_{0}$, $m\in\mathbb{Z}$ and $r\in\mathbb{Z}^{*}$.
Furthermore, we are able to deduce from the next proposition that $\Uh$ and $\Uv$ together generate the entire quantum toroidal algebra.
Figure \ref{Utor illustration} provides a simple illustration of $\Utor$ which highlights its generators and their $\mathbb{Z}$-grading, as well as the horizontal and vertical subalgebras.
\input{Utor_illustration}
\\

For any $i_{1},\dots,i_{p}\in I$ we define $\U_{i_{1}\dots i_{p}}$ to be the subalgebra of $\Utor$ generated by
$\lbrace \xpm_{\ell,m}, h_{\ell,r}, k_{\ell}^{\pm 1}, C^{\pm 1} : \ell = i_{1},\dots,i_{p}, \, m\in\mathbb{Z}, \, r\in\mathbb{Z}^{*} \rbrace$.
Then it is clear that each $\U_{i} \cong \UaffA$, and if $i\not= j$ then
\begin{align*}
    \U_{ij} \cong
    \begin{cases}
        \UaffA \times \UaffA &\mathrm{~if~} a_{ij}a_{ji} = 0, \\
        \UaffAA &\mathrm{~if~} a_{ij}a_{ji} = 1, \\
        \UaffC &\mathrm{~if~} a_{ij}a_{ji} = 2, \\
        \UaffG &\mathrm{~if~} a_{ij}a_{ji} = 3, \\
        \Utor &\mathrm{~in~types~} A_{1}^{(1)} \mathrm{~and~} A_{2}^{(2)}.
    \end{cases}
\end{align*}

\begin{prop} \label{toroidal generated by horizontal and Ui}
    For each $i\in I$, the quantum toroidal algebra is generated by $\Uh$, $\xpm_{i,\mp 1}$ and $C^{\pm 1}$.
\end{prop}

\begin{proof}
    Let $A$ be the subalgebra of $\Utor$ generated by $\Uh$, $\xpm_{i,\mp 1}$ and $C^{\pm 1}$.
    Our strategy is to first show that $\U_{i}$ is contained in $A$, and then to show that $\xpm_{j,\mp 1} \in A$ whenever $a_{ij}<0$, since the result then follows from the connectedness of the Dynkin diagram.
    Unpacking the formula in Definition~\ref{quantum toroidal algebra definition}, we see that $\phi^{\pm}_{i,0} = k_{i}^{\pm 1}$, $\phi^{\pm}_{i,\pm r} = 0$ for $r<0$, and
    \begin{align*}
        \phi^{\pm}_{i,\pm r} = k_{i}^{\pm 1} \sum_{\ell=1}^{r} \frac{(\pm 1)^{\ell}(q_{i}-q_{i}^{-1})^{\ell}}{\ell!} \sum_{\substack{r_{1}+\dots+r_{\ell}=r \\ \mathrm{all~}r_{j}>0}} h_{i,\pm r_{1}}\dots h_{i,\pm r_{\ell}}
    \end{align*}
    for $r>0$.
    This implies that $h_{i,-1} = C^{-1} k_{i} [\xp_{i,-1},\xm_{i,0}]$ and $h_{i,1} = C k_{i}^{-1} [\xp_{i,0},\xm_{i,1}]$.
    Thus by the relations of the quantum toroidal algebra, all $\xpm_{i,m}$ lie inside $A$.
    We also have
    \begin{equation*}
        [\xp_{i,\pm r},\xm_{i,0}] =
            \pm \frac{C^{\frac{r\mp r}{2}} k_{i}^{\pm 1} }{q_{i}-q_{i}^{-1}} \sum_{\ell=1}^{r} \frac{(\pm 1)^{\ell}(q_{i}-q_{i}^{-1})^{\ell}}{\ell!} \sum_{\substack{r_{1}+\dots+r_{\ell}=r \\ \mathrm{all~}r_{j}>0}} h_{i,\pm r_{1}}\dots h_{i,\pm r_{\ell}}
    \end{equation*}
    when $r>0$ and so all $h_{i,\pm r}\in A$ by induction.
    Therefore $\U_{i}\subset A$ and we conclude our proof by noting that $\xpm_{j,\mp 1} = \pm \frac{C^{\pm 1}}{[a_{ij}]_{i}} [h_{i,\mp 1},\xpm_{j,0}]$ whenever $a_{ij}<0$.
\end{proof}

\begin{cor} \label{toroidal generated by horizontal and vertical}
    The quantum toroidal algebra is generated by its horizontal and vertical subalgebras.
\end{cor}

Throughout the rest of this paper, we shall require the following standard automorphisms and anti-automorphisms of $\Utor$.

\begin{itemize}
    \item Every automorphism $\pi\in\Omega$ of the underlying affine Dynkin diagram gives rise to an automorphism $\Scal_{\pi}$ defined by
    \begin{align*}
        &\Scal_{\pi}(\xpm_{i,m}) = o_{i,\pi(i)}^{m}\xpm_{\pi(i),m}, \quad \Scal_{\pi}(k_{i}) = k_{\pi(i)}, \\
        &\Scal_{\pi}(h_{i,r}) = o_{i,\pi(i)}^{r}h_{\pi(i),r}, \quad \Scal_{\pi}(C) = C,
    \end{align*}
    which restricts to $S_{\pi}$ on $\Uh$.
    \item There is the anti-involution $\eta$ with
    \begin{align*}
        &\eta(\xpm_{i,m}) = \xpm_{i,-m}, \quad \eta(h_{i,r}) = -C^{r} h_{i,-r}, \quad \eta(k_{i}) = k_{i}^{-1}, \quad \eta(C) = C,
    \end{align*}
    which restricts to $\eta'$ on $\Uv$ and $\sigma$ on $\Uh$.
    \item For each $i\in I$ there is an automorphism $\X_{i}$ given by
    \begin{align*}
        &\X_{i}(\xpm_{j,m}) = o(j)^{\delta_{ij}} \xpm_{j,m\mp\delta_{ij}}, \quad \X_{i}(k_{j}) = C^{-\delta_{ij}} k_{j}, \\
        &\X_{i}(h_{j,r}) = h_{j,r}, \quad \X_{i}(C) = C.
    \end{align*}
    If $i\in I_{0}$ then $\X_{i}$ restricts to $\Xb_{i}$ on $\Uv$, while $\X_{0}$ restricts to the identity.
\end{itemize}

\subsection{Extended double affine braid groups} \label{Extended double affine braid group section}

Just as the quantum toroidal algebra $\Utor$ is in some sense formed by fusing together its horizontal and vertical quantum affine subalgebras, we can similarly define the extended double affine braid group $\Bdd$ by combining the Coxeter and Bernstein presentations for $\Bd$. \\

Recall from Section \ref{Extended affine braid group subsection} that $\Omega$ acts naturally on the affine braid group $\B = \langle T_{i} : i\in I\rangle$.
There is also a linear action of $\Omega$ on $P^{\vee}$ given by $\pi(\lambda_{i}^{\vee}) = \lambda_{\pi(i)}^{\vee}$, which preserves $\Pov \subset P^{\vee}$ and thus defines an action on $\lbrace X_{\beta} : \beta \in \Pov \rbrace$.
These actions are compatible with relations (\ref{first extended affine Bernstein}) and (\ref{second extended affine Bernstein}) for $N = \Pov$ (extended to all $i\in I$) and so the following is well-defined.

\begin{defn}
    The extended double affine braid group $\Bdd$ is generated by the affine braid group $\B = \langle T_{i} : i\in I\rangle$, the lattice $\lbrace X_{\beta} : \beta \in \Pov \rbrace$ and the group $\Omega$, subject to the relations
    \begin{itemize}
        \item $T_{i}X_{\beta} = X_{\beta}T_{i} \mathrm{~if~} (\beta,\alpha_{i}) = 0$,
        \item $T_{i}^{-1}X_{\beta}T_{i}^{-1} = X_{s_{i}(\beta)} \mathrm{~if~} (\beta,\alpha_{i}) = 1$,
        \item $\pi T_{i} \pi^{-1} = T_{\pi(i)}$,
        \item $\pi X_{\beta} \pi^{-1} = X_{\pi(\beta)}$.
    \end{itemize}
\end{defn}

\begin{rmk}
    \begin{enumerate}
        \item The action of $W$ on $\Pov$ in the definition above is with respect to the embedding $\Pov \hookrightarrow P^{\vee}$ of type $X_{n}^{(r)}$ rather than $Z_{n}^{(1)}$.
        \item Our group $\Bdd$ is the quotient of the $X,Y$-extended double affine Artin group of Ion and Sahi \cite{IS20}*{Chapter 9} by its central element $X_{\frac{1}{m}\delta}$.
    \end{enumerate}
\end{rmk}

It is clear that $\Bdd$ contains two extended affine braid subgroups which together generate the entire group: a horizontal subgroup $\Bh$ of type $X_{n}^{(r)}$ generated by $\B$ and $\Omega$, and a vertical subgroup $\Bv$ of type $Z_{n}^{(1)}$ generated by $T_{1},\dots,T_{n}$ and $\lbrace X_{\beta} : \beta \in \Pov \rbrace$.
We remark that there only exists an isomorphism between $\Bh$ and $\Bv$ which acts by the identity on $\B_{0} \cong \Bh \cap \Bv$ in the untwisted case.
\input{Bdd_illustration}
\\

From Section \ref{Extended affine braid group subsection} we know that $\Bh$ and $\Bv$ each have both Coxeter and Bernstein presentations -- Table~\ref{extended double affine braid group table} summarises our choice of notation.
\input{Bdd_notation_table}
In particular, for $\Bh$ we use the alternative Bernstein presentation of Remark \ref{alternative Bernstein presentation} so that while the $X_{\beta}$ satisfy relations (\ref{first extended affine Bernstein}) and (\ref{second extended affine Bernstein}) with $T_{0},\dots,T_{n}$, the $Y_{\mu}$ satisfy relations (\ref{first alternative extended affine Bernstein}) and (\ref{second alternative extended affine Bernstein}) with $T^{v}_{0},T_{1},\dots,T_{n}$.
Note that in all untwisted types, each $\pi_{i}$ and $\rho_{i}$ correspond to the same outer automorphism of the affine Dynkin diagram.
\\

We conclude this subsection with several automorphisms of $\Bdd$ which will be important in Section \ref{quantum toroidal automorphism section}.
For ease of notation, we restrict to the untwisted case since this is all we shall require.

\begin{itemize}
    \item There is an involution $\te$ which inverts $T_{1},\dots,T_{n}$ and interchanges $X_{\beta}$ and $Y_{\beta}$ for all $\beta\in\Pov$.
    It follows that $\te$ exchanges each $\pi_{i}$ and $\rho_{i}$, as well as $T_{0}$ and $(T^{v}_{0})^{-1}$.
    It is equal to the composition of the anti-involution $\e$ of Ion and Sahi \cite{IS20}*{Chapter 9} with the anti-automorphism that inverts every element.
    When restricted to the natural copy of the (non-extended) double affine braid group inside $\Bdd$, which is generated by $\B = \langle T_{0},\dots,T_{n} \rangle$ and $\lbrace X_{\beta} : \beta \in \Qov \rbrace$, this is the involution of Ion \cite{Ion03}*{Theorem 2.2}.
    \item There exists an involution $\gv$ inverting $T_{0},\dots,T_{n}$ and all $X_{\beta}$, while fixing each element of $\Omega$.
    Similarly, there is an involution $\gh = \te\circ\gv\circ\te$ which inverts $T^{v}_{0},T_{1},\dots,T_{n}$ and all $Y_{\mu}$ but fixes each element of $\Omega^{v}$.
\end{itemize}

\subsection{Braid group actions on quantum toroidal algebras} \label{Toroidal Braid Group Action Section}

Here we construct actions of the extended double affine braid groups on the quantum toroidal algebras.
We start with a simplified presentation of $\Utor$ involving finitely many generators and relations, which allows us to define automorphisms $\T_{i}$ for each $i\in I$.
Note that our proof relies upon a finite presentation of each subalgebra $\U_{ij}$ that comes from the Drinfeld-Jimbo presentation of the quantum affine algebras.
So in fact, our results extend to all quantum affinizations where the underlying Dynkin diagram has at most triple arrows, ie. $a_{ij}a_{ji}\leq 3$ for all distinct $i,j\in I$.
In particular, we exclude the quantum toroidal algebras of types $A_{1}^{(1)}$ and $A_{2}^{(2)}$ for the remainder of this paper.

\begin{lem} \label{lemma for new toroidal presentation}
    For $X$ of type $A_{2}$, $C_{2}$ or $G_{2}$, let $(a_{ij})_{i,j=1,2}$ be the corresponding finite type Cartan matrix and take $q_{1}$ and $q_{2}$ as in Section \ref{Preliminaries}.
    Define $\A_{X}$ to be the $\kk$-algebra with generators $\hat{x}^{\pm}_{i,0}$, $\hat{x}^{\pm}_{i,\mp 1}$, $\hat{k}_{i}^{\pm 1}$, $\hat{C}^{\pm 1}$ ($i=1,2$) and relations
    \begin{enumerate}[label={(\roman*)}]
    \item\label{AX 1} $\hat{C}^{\pm 1}$ central,
    \item\label{AX 2} $\displaystyle \hat{C}^{\pm 1}\hat{C}^{\mp 1} = \hat{k}_{i}^{\pm 1} \hat{k}_{i}^{\mp 1} = 1$,
    \item\label{AX 3} $\displaystyle [\hat{k}_{i},\hat{k}_{j}] = 0$,
    \item\label{AX 4} $\displaystyle \hat{k}_{i} \hat{x}^{\pm}_{j,m} \hat{k}_{i}^{-1} = q_{i}^{\pm a_{ij}} \hat{x}^{\pm}_{j,m}$,
    \item\label{AX 5} $\displaystyle [\hat{x}^{+}_{i,0},\hat{x}^{-}_{j,0}] = \frac{\delta_{ij}}{q_{i}-q_{i}^{-1}} (\hat{k}_{i} - \hat{k}_{i}^{-1})$,
    \item\label{AX 6} $\displaystyle [\hat{x}^{+}_{i,-1},\hat{x}^{-}_{i,1}] = \frac{\hat{C}^{-1}\hat{k}_{i} - \hat{C} \hat{k}_{i}^{-1}}{q_{i}-q_{i}^{-1}}$,
    \item\label{AX 7} $\displaystyle [\hat{x}^{+}_{i,0},\hat{x}^{-}_{j,1}] = [\hat{x}^{+}_{i,-1},\hat{x}^{-}_{j,0}] = 0$ whenever $i\not= j$,
    \item\label{AX 8} $\displaystyle [\hat{x}^{+}_{i,0},\hat{x}^{+}_{i,-1}]_{q_{i}^{2}} = [\hat{x}^{-}_{i,1},\hat{x}^{-}_{i,0}]_{q_{i}^{-2}} = 0$,
    \item\label{AX 9} $[\hat{x}^{+}_{1,0},\hat{x}^{+}_{2,-1}]_{q_{1}^{a_{12}}} + [\hat{x}^{+}_{2,0},\hat{x}^{+}_{1,-1}]_{q_{1}^{a_{12}}} = 0$,
    \item\label{AX 10} $[\hat{x}^{-}_{1,1},\hat{x}^{-}_{2,0}]_{q_{1}^{-a_{12}}} + [\hat{x}^{-}_{2,1},\hat{x}^{-}_{1,0}]_{q_{1}^{-a_{12}}} = 0$,
    \item\label{AX 11} $\displaystyle
    \sum_{s=0}^{1 - a_{ij}} (-1)^{s}
    {\begin{bmatrix}1 - a_{ij}\\s\end{bmatrix}}_{i}
    y_{i}^{s} y_{j} y_{i}^{1 - a_{ij}-s} = 0$  whenever $i\not= j$,
\end{enumerate}
for $(y_{i},y_{j}) = (\hat{x}^{\pm}_{i,0},\hat{x}^{\pm}_{j,0}),(\hat{x}^{\pm}_{i,\mp 1},\hat{x}^{\pm}_{j,0}),(\hat{x}^{\pm}_{i,0},\hat{x}^{\pm}_{j,\mp 1})$. \\

Then there is an algebra homomorphism $U_{q}(X^{(1)}) \rightarrow \A_{X}$ mapping
\begin{align*}
    C \mapsto \hat{C}, \quad
    \xpm_{i,0} \mapsto \hat{x}^{\pm}_{i,0}, \quad
    \xpm_{i,\mp 1} \mapsto \hat{x}^{\pm}_{i,\mp 1}, \quad
    k_{i} \mapsto \hat{k}_{i},
\end{align*}
for $i=1,2$.
\end{lem}

\begin{proof}
    Similar to Jing's isomorphism between the two realizations of the quantum affine algebra as outlined in Section \ref{quantum affine algebra subsection}, introduce elements
    \begin{align*}
        &\hat{k}_{0} = \hat{C} \hat{k}_{1}^{-a_{1}}\hat{k}_{2}^{-a_{2}}, \\
        &\hat{x}^{+}_{0,0} = \left[\hat{x}^{-}_{i_{h-1},0},\dots,\hat{x}^{-}_{i_{2},0},\hat{x}^{-}_{i_{1},1}\right]_{q^{\epsilon_{1}}\dots q^{\epsilon_{h-2}}} \hat{C} \hat{k}_{1}^{-a_{1}}\hat{k}_{2}^{-a_{2}}, \\
        &\hat{x}^{-}_{0,0} = a(-q)^{-\epsilon} \hat{C}^{-1} \hat{k}_{1}^{a_{1}}\hat{k}_{2}^{a_{2}} \left[\hat{x}^{+}_{i_{h-1},0},\dots,\hat{x}^{+}_{i_{2},0},\hat{x}^{+}_{i_{1},-1}\right]_{q^{\epsilon_{1}}\dots q^{\epsilon_{h-2}}},
    \end{align*}
    in $\A_{X}$ where $\underline{i} = (i_{1},\dots,i_{h-1})$ and $\underline{\epsilon} = (\epsilon_{1},\dots,\epsilon_{h-2})$ satisfy (\ref{Jing's isomorphism condition}).
    By relations \ref{AX 9}-\ref{AX 10} this is independent of the choice of sequences.
    Since the Drinfeld-Jimbo presentation of $U_{q}(X^{(1)})$ has finitely many relations, it is a finite check to prove that $\xpm_{i} \mapsto \hat{x}^{\pm}_{i,0}$ and $t_{i}^{\pm 1} \mapsto \hat{k}_{i}^{\pm 1}$ for $i = 0,1,2$ defines an algebra homomorphism $\xi : U_{q}(X^{(1)}) \rightarrow \A_{X}$.
    Note that it is immediate that $\xi$ preserves all relations with non-zero indices.
    Then to complete our proof we must express $C$, $\xpm_{1,\mp 1}$ and $\xpm_{2,\mp 1}$ in terms of the Drinfeld-Jimbo generators of $U_{q}(X^{(1)})$ and verify that the images under $\xi$ are equal to $\hat{C}$, $\hat{x}^{\pm}_{1,\mp 1}$ and $\hat{x}^{\pm}_{2,\mp 1}$ respectively (it is trivial that $C = t_{0}t_{1}^{a_{1}}t_{2}^{a_{2}}$ maps to $\hat{C} = \hat{k}_{0}\hat{k}_{1}^{a_{1}}\hat{k}_{2}^{a_{2}}$).
    See Appendix~\ref{Proof of lemma} for more details.
\end{proof}

Using this lemma we are able to prove the following simplified presentation for all quantum affinizations where the underlying Dynkin diagram has at most triple arrows.

\begin{prop} \label{simpler toroidal presentation}
    Let $\s$ be a symmetrizable Kac-Moody algebra with generalised Cartan matrix $(a_{ij})_{i,j\in I}$ and suppose that $a_{ij}a_{ji} \leq 3$ for all distinct $i,j\in I$.
    Then the quantum affinization $\widehat{U_{q}(\s)}$ has a finite presentation with generators $\lbrace \xpm_{i,0}, \xpm_{i,\mp 1}, k_{i}^{\pm 1}, C^{\pm 1} : i\in I \rbrace$ and relations
\begin{enumerate}[label={(\roman*)}]
    \item\label{simpler toroidal relation 1} $C^{\pm 1}$ central,
    \item $\displaystyle C^{\pm 1}C^{\mp 1} = k_{i}^{\pm 1} k_{i}^{\mp 1} = 1$,
    \item $\displaystyle [k_{i},k_{j}] = 0$,
    \item\label{simpler toroidal relation 4} $\displaystyle k_{i} \xpm_{j,m} k_{i}^{-1} = q_{i}^{\pm a_{ij}} \xpm_{j,m}$,
    \item\label{simpler toroidal relation 5} $\displaystyle [\xp_{i,0},\xm_{j,0}] = \frac{\delta_{ij}}{q_{i}-q_{i}^{-1}} (k_{i} - k_{i}^{-1})$,
    \item\label{simpler toroidal relation 6} $\displaystyle [\xp_{i,-1},\xm_{j,1}] = \delta_{ij} \frac{C^{-1}k_{i} - C k_{i}^{-1}}{q_{i}-q_{i}^{-1}}$,
    \item\label{simpler toroidal relation 7} $\displaystyle [\xp_{i,0},\xm_{j,1}] = [\xp_{i,-1},\xm_{j,0}] = 0$ whenever $i\not= j$,
    \item\label{simpler toroidal relation 8} $\displaystyle [\xp_{i,0},\xp_{i,-1}]_{q_{i}^{2}} = [\xm_{i,1},\xm_{i,0}]_{q_{i}^{-2}} = 0$,
    \item\label{simpler toroidal relation 9} $[\xp_{i,0},\xp_{j,-1}]_{q_{i}^{a_{ij}}} + [\xp_{j,0},\xp_{i,-1}]_{q_{i}^{a_{ij}}} = 0$ whenever $a_{ij} < 0$,
    \item\label{simpler toroidal relation 10} $[\xm_{i,1},\xm_{j,0}]_{q_{i}^{-a_{ij}}} + [\xm_{j,1},\xm_{i,0}]_{q_{i}^{-a_{ij}}} = 0$ whenever $a_{ij} < 0$,
    \item\label{simpler toroidal relation 11} $\displaystyle
    \sum_{s=0}^{1 - a_{ij}} (-1)^{s}
    {\begin{bmatrix}1 - a_{ij}\\s\end{bmatrix}}_{i}
    y_{i}^{s} y_{j} y_{i}^{1 - a_{ij}-s} = 0$  whenever $i\not= j$,
\end{enumerate}
for $(y_{i},y_{j}) = (x^{\pm}_{i,0},x^{\pm}_{j,0}),(x^{\pm}_{i,\mp 1},x^{\pm}_{j,0}),(x^{\pm}_{i,0},x^{\pm}_{j,\mp 1})$.
\end{prop}

\begin{proof}
    Let $\Hcal$ be the algebra generated by $\lbrace \hat{x}^{\pm}_{i,0}, \hat{x}^{\pm}_{i,\mp 1}, \hat{k}_{i}^{\pm 1}, \hat{C}^{\pm 1} : i\in I \rbrace$, subject to relations \ref{simpler toroidal relation 1}-\ref{simpler toroidal relation 11} above with hats over each generator.
    It is clear that
    \begin{align*}
    \hat{C} \mapsto C, \quad
    \hat{x}^{\pm}_{i,0} \mapsto \xpm_{i,0}, \quad
    \hat{x}^{\pm}_{i,\mp 1} \mapsto \xpm_{i,\mp 1}, \quad
    \hat{k}_{i} \mapsto k_{i},
\end{align*}
defines a homomorphism $f:\Hcal \rightarrow \Utor$.
Let us build its inverse out of morphisms $\U_{ij} \rightarrow \Hcal$.
By Lemma~\ref{lemma for new toroidal presentation}, we have for each pair of adjacent nodes $i,j\in I$ an algebra homomorphism $p_{ij} : \U_{ij} \rightarrow \Hcal$ given by
\begin{align*}
    C \mapsto \hat{C}, \quad
    \xpm_{\ell,0} \mapsto \hat{x}^{\pm}_{\ell,0}, \quad
    \xpm_{\ell,\mp 1} \mapsto \hat{x}^{\pm}_{\ell,\mp 1}, \quad
    k_{\ell} \mapsto \hat{k}_{\ell},
\end{align*}
for $\ell=i,j$.
Then $p_{i} = p_{ij}\vert_{\U_{i}} = p_{ji}\vert_{\U_{i}}$ is well-defined and independent of $j$, since our proof of Proposition~\ref{toroidal generated by horizontal and Ui} shows that $\U_{i}$ is generated by $\xpm_{i,0}$, $\xpm_{i,\mp 1}$, $k_{i}^{\pm 1}$ and $C^{\pm 1}$.
Furthermore, it is immediate from the presentation of $\Hcal$ that $[\mathrm{im}(p_{i}),\mathrm{im}(p_{j})] = 0$ whenever $a_{ij} = 0$.
So we see that
\begin{align*}
    C \mapsto \hat{C}, \quad
    \xpm_{i,0} \mapsto p_{i}(\xpm_{i,0}), \quad
    \xpm_{i,\mp 1} \mapsto p_{i}(\xpm_{i,\mp 1}), \quad
    k_{i} \mapsto p_{i}(k_{i}),
\end{align*}
for all $i\in I$ defines an algebra homomorphism $g:\Utor \rightarrow \Hcal$, since all relations of $\Utor$ are contained in some $\U_{ij}$.
We have $f\circ g = \mathrm{id}$ by checking on Drinfeld-Jimbo generators of each $\U_{ij}$ with $a_{ij}<0$, and $g\circ f = \mathrm{id}$ by checking on generators of $\Hcal$.
\end{proof}

\begin{rmk}
    \begin{enumerate}
        \item This result gives a finite Drinfeld new style presentation for the quantum toroidal algebra $\Utor$ in all untwisted and twisted types except $A_{1}^{(1)}$ and $A_{2}^{(2)}$, as well as for all untwisted quantum affine algebras $\Uaff$.
        \item The relations in Proposition \ref{simpler toroidal presentation} are a subset of those in the original definition for $\widehat{U_{q}(\s)}$ which only involve the generators $\xpm_{i,0}$, $\xpm_{i,\mp 1}$, $k_{i}^{\pm 1}$, $C^{\pm 1}$ for each $i\in I$.
        In particular, we do not have `shadows' of other relations appearing in our simplified presentation.
    \end{enumerate}
\end{rmk}

Recall that in Section \ref{Automorphisms of Uaff section} we obtained formulae for $\Tb_{i}(\xpm_{j,m})$ when $i,j\in I_{0}$ and $m = 0,\mp 1$.
Thanks to Proposition~\ref{simpler toroidal presentation}, these are enough to define the remaining automorphisms $\T_{i}$ that are required for our action of $\Bdd$ on $\Utor$.

\begin{prop} \label{Ti properties}
    For each $i\in I$ there exists an automorphism $\T_{i}$ of $\Utor$ such that
    \begin{itemize}
        \item $\T_{i} h = h \emph{\Tb}_{i}$ for all $i\in I$,
        \item $\T_{i} v = v \emph{\Tb}_{i}$ for all $i\in I_{0}$,
        \item $\T_{i}^{-1} = \eta\T_{i}\eta$ for all $i\in I$.
    \end{itemize}
\end{prop}
\begin{proof}
    For each $i\in I$, we define the morphism $\T_{i} : \Utor \rightarrow \Utor$ by
    \begin{align*}
        & \T_{i}(C) = C, \quad
        \T_{i}(k_{j}) = k_{j}k_{i}^{-a_{ij}}, \\
        & \T_{i}(\xp_{i,0}) = - \xm_{i,0} k_{i}, \quad
        \T_{i}(\xm_{i,0}) = - k_{i}^{-1} \xp_{i,0}, \\
        & \T_{i}(\xp_{i,-1}) = \sum_{s=0}^{2} (-1)^{s}q_{i}^{s}(\xm_{i,0})^{(s)}\xp_{i,-1}(\xm_{i,0})^{(2-s)} k_{i}, \\
        & \T_{i}(\xm_{i,1}) = k_{i}^{-1} \sum_{s=0}^{2} (-1)^{s}q_{i}^{-s}(\xp_{i,0})^{(2-s)}\xm_{i,1}(\xp_{i,0})^{(s)}, \\
        & \T_{i}(\xp_{j,m}) = \sum_{s=0}^{-a_{ij}} (-1)^{s} q_{i}^{-s} (\xp_{i,0})^{(-a_{ij}-s)} \xp_{j,m} (\xp_{i,0})^{(s)} \mathrm{~if~} i\not= j, \\
        & \T_{i}(\xm_{j,m}) = \sum_{s=0}^{-a_{ij}} (-1)^{s} q_{i}^{s} (\xm_{i,0})^{(s)} \xm_{j,m} (\xm_{i,0})^{(-a_{ij}-s)} \mathrm{~if~} i\not= j.
    \end{align*}
    To verify that $\T_{i}$ is a well-defined homomorphism, we need to show that it preserves every relation in our simplified presentation of $\Utor$.
    For each $j,\ell\in I$ consider the relations lying inside $\U_{j\ell}$.
    These are also relations of the subalgebra $\U_{ij\ell}$, to which we may restrict since from the formulae above it is preserved by $\T_{i}$.
    Note that $\U_{ij\ell}$ is isomorphic to the quantum affinization of the quantum group associated to the full Dynkin subdiagram $D_{ij\ell}$ on the nodes $i,j,\ell$. \\
    
    If $D_{ij\ell}$ is a subdiagram of some finite Dynkin diagram, then $\U_{ij\ell}$ is isomorphic to a direct product of quantum affine algebras.
    Furthermore, $\T_{i}\vert_{\U_{ij\ell}}$ acts by $\Tb_{i}$ on the factor containing $\U_{i}$ and by the identity on all other factors.
    Since this is an automorphism of $\U_{ij\ell}$, we see that $\T_{i}$ preserves all relations lying inside $\U_{j\ell}$. \\

    Otherwise, $D_{ij\ell}$ is the Dynkin diagram of type $A_{2}^{(1)}$, $C_{2}^{(1)}$, $G_{2}^{(1)}$, $A_{4}^{(2)}$, $D_{3}^{(2)}$ or $D_{4}^{(3)}$ and $\U_{ij\ell}$ is isomorphic to the corresponding quantum toroidal algebra.
    Since by definition $\T_{i}$ restricts to $\Tb_{i}$ on the horizontal subalgebra $(\U_{ij\ell})_{h}$, it preserves all relations lying inside it.
    Any other relation can then be obtained from one of these by applying $\X_{j}$, $\X_{\ell}$ or $\X_{j}\X_{\ell}$, which commute with $\T_{i}$ since $i\not= j,\ell$. \\
    
    Hence we may conclude that $\T_{i}$ respects all relations in our simplified presentation, and is thus an algebra homomorphism.
    The first two bullet points in the statement of the proposition are then immediate from the formulae for $\Tb_{i}$ on the Drinfeld-Jimbo and Drinfeld new generators of $\Uaff$ in Section \ref{Automorphisms of Uaff section}. \\

    To show that $\T_{i}$ is an automorphism with inverse $\eta\T_{i}\eta$, it suffices to check this on the invariant subspace $\U_{ij}$ for each $j\not= i$.
    If $a_{ij}<0$ then $\U_{ij}$ is isomorphic to $\UaffAA$, $\UaffC$ or $\UaffG$, and $\T_{i}$ and $\eta$ restrict to $\Tb_{i}$ and $\eta'$.
    If $a_{ij} = 0$ then $\U_{ij} \cong \U_{i}\times\U_{j}$ and $\T_{i}$ and $\eta$ restrict to $\Tb_{i}\times\mathrm{id}$ and $\eta'\times\eta'$.
    In either case, since $(\Tb_{i})^{-1} = \eta'\Tb_{i}\eta'$ our proof is complete.
\end{proof}

We now have all of the automorphisms required to define our braid group action on $\Utor$.
However, in type $A_{2n}^{(1)}$ we are forced to consider a slightly modified version of $\Bdd$.
First, we must have that $\pi_{1}\in\Omega$ has order $4n+2$ rather than $2n+1$.
This is because, as discussed in Section \ref{Preliminaries}, there is no length function on the affine Dynkin diagram and so
\begin{align*}
    &\Scal_{\pi_{1}}^{2n+1}(\xpm_{i,m}) = (-1)^{m}\xpm_{i,m}, \quad \Scal_{\pi_{1}}^{2n+1}(k_{i}) = k_{i}, \\
    &\Scal_{\pi_{1}}^{2n+1}(h_{i,r}) = (-1)^{r}h_{i,r}, \quad \Scal_{\pi_{1}}^{2n+1}(C) = C,
\end{align*}
has order two.
Let $\zeta_{i}$ be the automorphism mapping each $\xpm_{i,m} \mapsto -\xpm_{i,m}$ and fixing the other generators.
Then we have
\begin{align*}
    \Scal_{\pi_{1}} \zeta_{i} \Scal_{\pi_{1}}^{-1} = \zeta_{\pi_{1}(i)}, \quad
    \Scal_{\pi_{1}} \X_{2n} \Scal_{\pi_{1}}^{-1} = \zeta_{0}\X_{0}, \quad
    \T_{0}^{-1} \X_{0} \T_{0}^{-1} = \zeta_{0}\X_{2n}\X_{0}^{-1}\X_{1},
\end{align*}
and we adjust the relations of $\Bdd$ accordingly.
The involutions $\te$, $\gv$ and $\gh$ extend naturally to our modified braid group, and our results are not otherwise impacted.
The proof of the next theorem is virtually the same as for the other cases (we shall not include the minor differences) and there is only a slight change in Lemma~\ref{lemma for psi theorem}.

\begin{thm} \label{toroidal braid group action theorem}
    The extended double affine braid group $\Bdd$ acts on the quantum toroidal algebra $\Utor$ via
    $T_{i} \mapsto \T_{i}$ for all $i\in I$,
    $X_{\omega_{i}^{\vee}} \mapsto \Z_{\omega_{i}^{\vee}} := \X_{i}\X_{0}^{-a_{i}}$ for all $i\in I_{0}$,
    and $\pi \mapsto \Scal_{\pi}$ for all $\pi\in \Omega$.
\end{thm}
\begin{proof}
    The relations between $\T_{i}$ and $\Z_{\beta}$ follow from the Coxeter relations between $\T_{i}$ and $\X_{j}$, namely that $\T_{i}\X_{j} = \X_{j}\T_{i}$ whenever $i\not=j$ and $\T_{i}^{-1}\X_{i}\T_{i}^{-1} = \X_{i}\prod_{j\in I}\X_{j}^{-a_{ij}}$.
    Commutativity of $\T_{i}$ and $\X_{j}$ for $i\not= j$ is clear from the definitions, while the other relation is checked on each $\U_{\ell}$ by restricting to $\U_{i\ell}$ and applying Theorem~\ref{Beck's affine result}, since $\U_{i\ell}$ is a product of quantum affine algebras. \\
    
    To verify the braid relation between $\T_{i}$ and $\T_{j}$ on elements of $\U_{\ell}$, we restrict to the invariant subalgebra $\U_{ij\ell}$.
    Similarly to our proof of Proposition~\ref{Ti properties}, if $D_{ij\ell}$ is a subdiagram of a finite Dynkin diagram then $\U_{ij\ell}$ is a product of quantum affine algebras and we can use the braid relation between $\Tb_{i}$ and $\Tb_{j}$ from Theorem~\ref{Beck's affine result}.
    Otherwise, $\U_{ij\ell}$ is isomorphic to the quantum toroidal algebra of untwisted type $A_{2}$, $C_{2}$ or $G_{2}$. We already have the braid relation on $\xpm_{\ell,0}$ and $k_{\ell}^{\pm 1}$ using $\T_{i}h = h\Tb_{i}$ and $\T_{j}h = h\Tb_{j}$.
    Applying $\X_{\ell}$, which commutes with $\T_{i}$ and $\T_{j}$ since $\ell\not= i,j$, we derive the braid relation on all of $\U_{\ell}$.
    The remaining relations follow from the definitions without much difficulty.
\end{proof}

\begin{rmk} \label{braid restriction remark}
    \begin{enumerate}
        \item Our extended double affine braid group action restricts to both an action of $\B_{h}$ on $\Uh$ and an action of $\B_{v}$ on $\Uv$, each of which coincides with Lusztig and Beck's action of the extended affine braid group on the quantum affine algebra.
        \item In their PhD thesis, motivated by trying to obtain a Damiani-Beck style isomorphism on the toroidal level, Mounzer \cite{Mounzer22} provides a topological braid group action on a certain completion of $\Utor$ (verifying the quantum Serre relations is a work in progress in some cases).
        We note that this action does not restrict to the quantum toroidal algebra and is thus distinct from our results.
    \end{enumerate}
\end{rmk}

It is worth highlighting that these results extend naturally to the quantum affinizations $\widehat{U_{q}(\s)}$ considered in Proposition \ref{simpler toroidal presentation}, namely those with $a_{ij}a_{ji}\leq 3$ for all distinct $i,j\in I$.
In particular, for each $i\in I$ there exists an automorphism $\T_{i}$ of $\widehat{U_{q}(\s)}$ defined exactly as in Proposition \ref{Ti properties}, with inverse $\T_{i}^{-1} = \eta\T_{i}\eta$, which restricts to $\Tb_{i}$ on the horizontal copy of $U_{q}(\s)$.
Furthermore, for some appropriate generalisation of $\Bdd$ we obtain a braid group action as in Theorem \ref{toroidal braid group action theorem}.
In each case, the proofs are the same as above.

\begin{defn}
    For any generalised Cartan matrix $(a_{ij})_{i,j\in I}$ we define $\widehat{\B}$ to be the group generated by $\lbrace T_{i},X_{i} : i\in I \rbrace$ and the automorphism group $\Omega$ of the associated Dynkin diagram, with relations
    \begin{itemize}
        \item $T_{i}T_{j}T_{i}\ldots = T_{j}T_{i}T_{j}\ldots$ whenever $a_{ij}a_{ji}\leq 3$, where there are $a_{ij}a_{ji} + 2$ factors on each side,
        \item $X_{i}X_{j} = X_{j}X_{i}$,
        \item $T_{i}X_{j} = X_{j}T_{i}$ whenever $i\not= j$,
        \item $T_{i}^{-1}X_{i}T_{i}^{-1} = X_{i} \prod_{j\in I}X_{j}^{-a_{ij}}$,
        \item $\pi T_{i} \pi^{-1} = T_{\pi(i)}$,
        \item $\pi X_{i} \pi^{-1} = X_{\pi(i)}$,
    \end{itemize}
    for all $i,j\in I$ and $\pi\in\Omega$.
\end{defn}

Note that the extended double affine braid group $\Bdd$ embeds inside the corresponding $\widehat{\B}$ by sending $T_{i} \mapsto T_{i}$, $X_{\omega_{i}^{\vee}} \mapsto X_{i}X_{0}^{-a_{i}}$ and $\pi \mapsto \pi$ for each $i\in I$ and $\pi\in\Omega$.
\\

When the underlying Dynkin diagram possesses a length function $o$, by defining $\X_{i}$ and $\Scal_{\pi}$ exactly as for $\Utor$ in Section \ref{quantum toroidal algebras subsection}, we obtain the following `braid group action' on $\widehat{U_{q}(\s)}$.

\begin{thm}
    The group $\widehat{\B}$ acts on the quantum affinization $\widehat{U_{q}(\s)}$ via $T_{i} \mapsto \T_{i}$ and $X_{i} \mapsto \X_{i}$ for all $i\in I$, and $\pi \mapsto \Scal_{\pi}$ for all $\pi\in \Omega$.
\end{thm}

If instead no such $o$ exists and the Dynkin diagram contains an odd length cycle, this should instead hold for a modified version of $\widehat{\B}$ as was the case in type $A_{2n}^{(1)}$.

\section[Automorphisms and anti-automorphisms of quantum toroidal \texorpdfstring{\\}{} algebras]{Automorphisms and anti-automorphisms of \\ quantum toroidal algebras} \label{quantum toroidal automorphism section}

We now look to construct certain automorphisms and anti-involutions of $\Utor$ which exchange the horizontal and vertical subalgebras.
Our current method relies on certain properties of the root system, in particular the $a_{i}$ labels, and so for the remainder of this paper we restrict our focus to the (untwisted) simply laced cases.
We shall return to the other types in future work.
For notational simplicity, we will henceforth identify elements of $\Bdd$ with the corresponding automorphisms of $\Utor$ from Theorem~\ref{toroidal braid group action theorem}.
We shall also write $X_{i}$ for $X_{\omega_{i}^{\vee}}$ and $Y_{i}$ for $Y_{\omega_{i}^{\vee}}$ for each $i\in I_{0}$. \\

Our approach is roughly as follows.
We can in some sense build $\Utor$ out of the copy of the finite quantum group $\Uq$ lying inside $\Uh\cap\Uv$ and the braid group action from Theorem~\ref{toroidal braid group action theorem}.
Twisting the action by certain automorphisms of $\Bdd$ obtains a different `twisted' set of generators for $\Utor$.
Then mapping the standard generators to their twisted counterparts gives our desired (anti-)automorphisms. \\

More specifically, each generator of our simplified presentation of $\Utor$ from Proposition~\ref{simpler toroidal presentation} (other than $C^{\pm 1}$) can easily be written as $b(z)$ for some $b\in\Bdd$ and $z\in\Uq$.
For all $\xpm_{i,0}$ and $k_{i}^{\pm 1}$ with $i\in I_{0}$ we may set $b=1$, and for the other generators we have
\begin{itemize}
    \item $\xpm_{i,\mp 1} = o(i) X_{i}(\xpm_{i,0})$ for each $i\in I_{0}$,
    \item $\xpm_{0,0} = \pi_{j}^{-1}(\xpm_{j,0})$ for any $j\in\Imin\setminus\lbrace 0\rbrace$,
    \item $\xpm_{0,\mp 1} = o(0) \pi_{j}^{-1} X_{j}(\xpm_{j,0})$ for any $j\in\Imin\setminus\lbrace 0\rbrace$,
    \item $k_{0}^{\pm 1} = \pi_{j}^{-1}(k_{j}^{\pm 1})$ for any $j\in\Imin\setminus\lbrace 0\rbrace$,
\end{itemize}
outside of type $E_{8}^{(1)}$, where since $\Omega$ is trivial we instead write
\begin{itemize}
    \item $\xpm_{i,\mp 1} = o(i) X_{i}(\xpm_{i,0})$ for each $i\in I_{0}$,
    \item $\xpm_{0,0} = T_{1}T_{0}(\xpm_{1,0})$,
    \item $\xpm_{0,\mp 1} = o(0) X_{\beta} T_{1}T_{0}(\xpm_{1,0})$ whenever $(\beta,\alpha_{0}) = 1$,
    \item $k_{0}^{\pm 1} = T_{1}T_{0}(k_{1}^{\pm 1})$.
\end{itemize}
Recall the involution $\te$ of $\Bdd$ from Section \ref{Extended double affine braid group section}.
For each $\xpm_{i,m} = b(z)$ above let $\xbpm_{i,m} = \te(b)(z)$,
and for each $k_{i}^{\pm 1} = b(z)$ define $\kb_{i}^{\pm 1} = \te(b)(z^{-1})$.
In particular,
\begin{align*}
    \kb_{i}^{\pm 1} = k_{i}^{\mp 1}, \quad
    \xbpm_{i,0} = \xpm_{i,0}, \quad
    \xbpm_{i,\mp 1} = o(i) Y_{i}(\xpm_{i,0}),
\end{align*}
for all $i\in I_{0}$, and outside of $E_{8}^{(1)}$ we have
\begin{align*}
    \kb_{0}^{\pm 1} = \rho^{-1}_{j}(k_{j}^{\mp 1}), \quad
    \xbpm_{0,0} = \rho^{-1}_{j}(\xpm_{j,0}), \quad
    \xbpm_{0,\mp 1} = o(0) \rho^{-1}_{j} Y_{j}(\xpm_{j,0}),
\end{align*}
for any $j\in\Imin\setminus\lbrace 0\rbrace$.
In type $E_{8}^{(1)}$ these are replaced with
\begin{align*}
    &\kb_{0}^{\pm 1} = T_{1}^{-1}(T^{v}_{0})^{-1}(k_{1}^{\mp 1}), \\
    &\xbpm_{0,0} = T_{1}^{-1}(T^{v}_{0})^{-1}(\xpm_{1,0}), \\
    &\xbpm_{0,\mp 1} = o(0) Y_{\beta} T_{1}^{-1}(T^{v}_{0})^{-1}(\xpm_{1,0}),
\end{align*}
whenever $(\beta,\alpha_{0}) = 1$.
If we also define $\Cb^{\pm 1} = (k_{0}^{a_{0}}\dots k_{n}^{a_{n}})^{\mp 1}$ then the following theorem shows that mapping generators to their bold counterparts extends to an anti-involution of $\Utor$ which exchanges $\Uh$ and $\Uv$ (up to a twist by $\sigma$).
As an immediate corollary, we can deduce the injectivity of $h$ from that of $v$.

\begin{thm} \label{psi theorem}
    There is an anti-involution $\psi$ of $\Utor$ sending
    \begin{align*}
        \xpm_{i,m} \mapsto \xbpm_{i,m}, \quad
        k_{i} \mapsto \kb_{i}, \quad
        C \mapsto \Cb,
    \end{align*}
    for all $i\in I$ and $m=0,\mp 1$, determined by the conditions $\psi v = h \sigma$ and $\psi h = v \sigma$.
\end{thm}

A brief technical lemma gives various identities required for the proof of the above.
Note that in type $A_{2n}^{(1)}$ we restrict to $\rho = \rho_{1}$ for (\ref{lemma (4)}).

\begin{lem}\label{lemma for psi theorem}
\begin{itemize}
        \item $Y_{i}(\xbpm_{j,0}) = \xbpm_{j,0}$ and $Y_{i}(\kb_{j}^{\pm 1}) = \kb_{j}^{\pm 1}$ for all distinct $i,j \in I_{0}, \hfill \refstepcounter{equation}(\theequation)\label{lemma (1)}$
        \item $\xbpm_{i,m} = h\sigma(\xpm_{i,m})$, $\kb_{i}^{\pm 1} = h\sigma(k_{i}^{\pm 1})$ and $\Cb^{\pm 1} = h\sigma(C^{\pm 1})$ for all $i \in I_{0}$ and $m=0,\mp 1, \hfill \refstepcounter{equation}(\theequation)\label{lemma (2)}$
        \item $\xbpm_{i,0} = v\sigma(\xpm_{i})$ and $\kb_{i}^{\pm 1} = v\sigma(t_{i}^{\pm 1})$ for all $i\in I, \hfill \refstepcounter{equation}(\theequation)\label{lemma (3)}$
        \item $\rho(\xbpm_{i,m}) = o_{i,\rho(i)}^{m}\xbpm_{\rho(i),m}$ and $\rho(\kb_{i}^{\pm 1}) = \kb_{\rho(i)}^{\pm 1}$ for all $i\in I$, $m=0,\mp 1$ and $\rho\in\Omega^{v}. \hfill \refstepcounter{equation}(\theequation)\label{lemma (4)}$
    \end{itemize}
\end{lem}
\begin{proof}
    We know from Proposition~\ref{Ti properties} that $T_{i}h = h \Tb_{i} = h \sigma \Tb_{i}^{-1} \sigma$ for all $i\in I$, and it is immediate from the definitions that $\pi h = h S_{\pi} = h \sigma S_{\pi} \sigma$ for each $\pi\in\Omega$.
    Each $Y_{\beta}$ can be written as $\pi T_{i_{1}}^{\pm 1}\dots T_{i_{s}}^{\pm 1}$ and so as $\sigma^{2}$ is the identity,
    \begin{align}\label{Y on Uh}
        Y_{\beta} h = h \sigma S_{\pi}\Tb_{i_{1}}^{\mp 1}\dots\Tb_{i_{s}}^{\mp 1} \sigma = h \sigma \Xb_{\beta} \sigma.
    \end{align}
    Note that (\ref{lemma (2)}) for $\xbpm_{i,0}$, $\kb_{i}^{\pm 1}$ and $\Cb^{\pm 1}$ is trivial, as is (\ref{lemma (3)}) when $i\in I_{0}$.
    Using equation (\ref{Y on Uh}) we can then deduce (\ref{lemma (1)}) and the remainder of (\ref{lemma (2)}).
    As mentioned in Remark~\ref{braid restriction remark}, $\Bv$ acts on $\Uv$ via Beck's extended affine braid group action.
    In particular, each $\rho_{j} v = v S_{\rho_{j}}$ and so outside of $E_{8}^{(1)}$,
    \begin{align*}
        & \xbpm_{0,0} = \rho^{-1}_{j} v(\xpm_{i}) = v S^{-1}_{\rho_{j}} (\xpm_{j}) = v(\xpm_{0}) = v \sigma (\xpm_{0}), \\
        & \kb_{0}^{\pm 1} = \rho^{-1}_{j} v(t_{j}^{\mp 1}) = v S^{-1}_{\rho_{j}} (t_{j}^{\mp 1}) = v(t_{0}^{\mp 1}) = v \sigma (t_{0}^{\pm 1}),
    \end{align*}
    for any $j\in\Imin\setminus\lbrace 0\rbrace$.
    On the other hand, in type $E_{8}^{(1)}$ we have
    \begin{align*}
        & \xbpm_{0,0} = T_{1}^{-1}(T^{v}_{0})^{-1}v(\xpm_{1,0}) = v\Tb_{1}^{-1}\Tb_{0}^{-1}(\xpm_{1,0}) = v(\xpm_{0}) = v\sigma(\xpm_{0}), \\
        & \kb_{0}^{\pm 1} = T_{1}^{-1}(T^{v}_{0})^{-1}v(t_{1}^{\mp 1}) = v\Tb_{1}^{-1}\Tb_{0}^{-1}(t_{i}^{\mp 1}) = v(t_{0}^{\mp 1}) = v\sigma(t_{0}^{\pm 1}).
    \end{align*}
    Then for all $\rho\in\Omega^{v}$ and $i\in I$,
    \begin{align*}
        & \rho(\xbpm_{i,0}) = \rho v(\xpm_{i}) = v S_{\rho} (\xpm_{i}) = v(\xpm_{\rho(i)}) = \xbpm_{\rho(i),0}, \\
        & \rho(\kb_{i}^{\pm 1}) = \rho v(t_{i}^{\mp 1}) = v S_{\rho} (t_{i}^{\mp 1}) = v(t_{\rho(i)}^{\mp 1}) = \kb_{\rho(i)}^{\pm 1}.
    \end{align*}
    It follows that whenever $i$, $\rho(i)$ and $\rho(0)$ are non-zero,
    \begin{align*}
        \rho(\xbpm_{i,\mp 1}) &= o(i) \rho Y_{i}(\xbpm_{i,0}) = o(i) Y_{\rho(i)} Y_{\rho(0)}^{-a_{i}} \rho(\xbpm_{i,0}) \\
        &= o(i) Y_{\rho(i)} Y_{\rho(0)}^{-a_{i}} (\xbpm_{\rho(i),0}) = o_{i,\rho(i)}\xbpm_{\rho(i),\mp 1},
    \end{align*}
    and the cases $i=0$ and $\rho(0) = 0$ are trivial.
    Outside of type $A_{2n}^{(1)}$ it is immediate from the definitions that $\rho(\xbpm_{\rho^{-1}(0),\mp 1}) = o_{i,0}\xbpm_{0,\mp 1}$.
    In type $A_{2n}^{(1)}$, from the other identities this is equivalent to $\rho_{1}^{2}(\xbpm_{2n,\mp 1}) = \xbpm_{1,\mp 1}$,
    and the equality $\rho_{1}^{2} Y_{2n} Y_{2n-1}^{-1} = \zeta_{0}\zeta_{1} Y_{1} \rho_{1}^{2}$ implies that $\rho_{1}^{2}(\xbpm_{2n,\mp 1}) = \zeta_{0}\zeta_{1}(\xbpm_{1,\mp 1})$.
    Then using (\ref{lemma (2)}) and $\Xb_{1} = S_{\pi_{1}}\Tb_{n}\dots\Tb_{1}$ we can obtain an explicit expression for $\xbpm_{1,\mp 1}$ in terms of $\lbrace \xpm_{i,0}, k_{i}^{\pm 1} : i\in I\rbrace$, from which we deduce that $\zeta_{0}\zeta_{1}(\xbpm_{1,\mp 1}) = \xbpm_{1,\mp 1}$.
\end{proof}

\begin{proof}[Proof of Theorem~{\upshape\ref{psi theorem}}]
    To show that $\psi$ is an anti-homomorphism, we must check that relations \ref{simpler toroidal relation 1}-\ref{simpler toroidal relation 11} of Proposition~\ref{simpler toroidal presentation} still hold if we reverse the order of multiplication and replace each generator with its image.
    Denote these modified relations by \textbf{(i)}-\textbf{(xi)}.
    Every relation with indices in $I_{0}$ follows immediately from the Drinfeld new presentation of $\Uh$ using (\ref{lemma (2)}).
    Moreover, relations involving only $\xbpm_{i,0}$ and $\kb_{i}^{\pm 1}$ terms follow from the Drinfeld-Jimbo presentation of $\Uv$ by (\ref{lemma (3)}). \\
    
    When $|\Omega|>2$ all other relations can be reached by applying some $\rho\in\Omega^{v}$ and (\ref{lemma (4)}) to a relation with indices in $I_{0}$.
    In type $E_{7}^{(1)}$, the same reasoning gives all relations apart from those involving indices $\lbrace i,j\rbrace = \lbrace 0,6\rbrace$ which are not contained in $\Uv$.
    If $(\beta,\alpha_{0}) = 0$ then $(\rho_{6}(\beta),\alpha_{6}) = 0$ so by (\ref{lemma (4)}) and (\ref{Y on Uh}),
    \begin{align*}
        Y_{\beta}(\kb_{0}^{\pm 1}) &=
        Y_{\beta}\rho_{6}(\kb_{6}^{\pm 1}) =
        \rho_{6} Y_{\rho_{6}(\beta)} h\sigma(k_{6}^{\pm 1}) \\
        &=
        \rho_{6} h\sigma X_{\rho_{6}(\beta)} (k_{6}^{\pm 1}) =
        \rho_{6} h\sigma(k_{6}^{\pm 1}) \\
        &=
        \kb_{0}^{\pm 1},
    \end{align*}
    and $Y_{\beta}(\xbpm_{0,m}) = \xbpm_{0,m}$ for $m=0,\mp 1$ via similar equalities.
    Therefore, taking for example $Y_{\mu} = Y_{4}^{-1}Y_{5}Y_{6}$ we can obtain the remaining relations as follows.
    \begin{enumerate}
        \item[\textbf{(iv)}]
        For $m=\mp 1$ apply $Y_{\mu}$ and $\rho_{6} Y_{\mu}$ to the $m=0$ cases.
        \item[\textbf{(vi)}-\textbf{(vii)}]
        Apply $Y_{\mu}$, $\rho_{6} Y_{\mu}$ and $Y_{\mu} \rho_{6} Y_{\mu}$ to the corresponding relations in \textbf{(v)}.
        \item[\textbf{(xi)}]
        Apply $Y_{\mu}$, $\rho_{6} Y_{\mu}$ and $Y_{\mu} \rho_{6} Y_{\mu}$ to the cases involving $\xbpm_{0,0}$ and $\xbpm_{6,0}$.
    \end{enumerate}
    In type $E_{8}^{(1)}$ we proceed in an analogous fashion.
    Table~\ref{E8 Y table} lists for each $i\in I_{0}$ some $Y_{\mu_{i}}$ which fixes $\xbpm_{i,0}$ and $\kb_{i}^{\pm 1}$ and sends $\xbpm_{0,0}\mapsto o(0)\xbpm_{0,\mp 1}$, and some $Y_{\mu'_{i}}$ which fixes $\xbpm_{0,0}$ and $\kb_{0}^{\pm 1}$ and maps $\xbpm_{i,0}\mapsto o(i)\xbpm_{i,\mp 1}$.
    \input{E8_table}
    Let us verify the conditions on $Y_{\mu_{i}}$ and $Y_{\mu'_{i}}$.
    In all cases $Y_{\mu_{i}}(\xbpm_{0,0}) = o(0)\xbpm_{0,\mp 1}$ is immediate by definition, and (\ref{lemma (1)}) gives
    \begin{align*}
        Y_{\mu_{i}}(\kb_{i}^{\pm 1}) = \kb_{i}^{\pm 1}, \quad
        Y_{\mu_{i}}(\xbpm_{i,0}) = \xbpm_{i,0}, \quad
        Y_{\mu'_{i}}(\xbpm_{i,0}) = o(i)\xbpm_{i,\mp 1}.
    \end{align*}
    The other properties follow from $Y_{\mu'_{i}}T_{1}^{-1}(T^{v}_{0})^{-1} = T_{1}^{-1}(T^{v}_{0})^{-1}Y_{\mu'_{i}}$ when $i\not= 1$, and $Y_{\mu'_{1}}T_{1}(T^{v}_{0}) = T_{1}^{-1}(T^{v}_{0})^{-1}Y_{2}Y_{7}^{-1}$.
    We can then obtain the remaining relations for $\psi$ to be an anti-homomorphism as follows.
    \begin{itemize}
        \item[\textbf{(iv)}]
        For $m=\mp 1$ apply $Y_{\mu_{i}}$ and $Y_{\mu'_{i}}$ to the $m=0$ cases.
        \item[\textbf{(vi)}-\textbf{(vii)}]
        Apply $Y_{\mu_{i}}$, $Y_{\mu'_{i}}$ and $Y_{\mu_{i}} Y_{\mu'_{i}}$ to the corresponding relations in \textbf{(v)}.
        \item[\textbf{(ix)}-\textbf{(x)}] The case $\lbrace i,j\rbrace = \lbrace 0,1\rbrace$ comes from applying both sides of
        \begin{align*}
            o(0) Y_{\mu_{1}} T_{1}^{-1} =
            -o(1) Y_{\mu'_{1}} (T^{v}_{0})^{-1} Y_{7}^{2}Y_{8}^{-1} (T^{v}_{0})^{-1} T_{1}^{-1}
        \end{align*}
        to $\xbpm_{0,0}$, since using Lemma~\ref{lemma for psi theorem} and Remark~\ref{braid restriction remark} we have
        \begin{gather*}
            \xb^{+}_{0,0} \xmapsto{T_{1}^{-1}}
            [\xb^{+}_{0,0},\xb^{+}_{1,0}]_{q^{-1}}
            \xmapsto{o(0) Y_{\mu_{1}}}
            [\xb^{+}_{0,-1},\xb^{+}_{1,0}]_{q^{-1}} \\
            \xb^{+}_{0,0} \xmapsto{(T^{v}_{0})^{-1} T_{1}^{-1}}
            \xb^{+}_{1,0} \xmapsto{(T^{v}_{0})^{-1}Y_{7}^{2}Y_{8}^{-1}}
            [\xb^{+}_{1,0},\xb^{+}_{0,0}]_{q^{-1}}
            \xmapsto{o(1) Y_{\mu'_{1}}}
            [\xb^{+}_{1,-1},\xb^{+}_{0,0}]_{q^{-1}}
        \end{gather*}
        and similarly for $\xb^{-}_{0,0}$.
        \item[\textbf{(xi)}]
        Apply $Y_{\mu_{i}}$, $Y_{\mu'_{i}}$ and $Y_{\mu_{i}} Y_{\mu'_{i}}$ to the cases involving $\xbpm_{0,0}$ and $\xbpm_{i,0}$.
    \end{itemize}
    We have therefore shown that $\psi$ is an anti-homomorphism.
    The conditions $\psi v = h\sigma$ and $\psi h = v\sigma$ are then immediate from (\ref{lemma (2)}) and (\ref{lemma (3)}), and moreover determine $\psi$ since $\Uh$ and $\Uv$ generate $\Utor$.
    Furthermore, it also follows that $\psi^{2} = \mathrm{id}$ on $\Uh$ and $\Uv$ and so $\psi$ is in fact an anti-involution.
\end{proof}

Figure~\ref{illustrations} contains simple illustrations of the quantum toroidal algebra which highlight where elements of the two generating sets
$\lbrace \xpm_{i,0}, \xpm_{i,\mp 1}, k_{i}^{\pm 1}, C^{\pm 1} : i\in I \rbrace$ and
$\lbrace \xbpm_{i,0}, \xbpm_{i,\mp 1}, \kb_{i}^{\pm 1}, \Cb^{\pm 1} : i\in I \rbrace$
lie inside $\Utor$.
\input{two_finite_Utor_illustrations}
In particular, we see how the bold generators in some sense give $\Utor$ as a quantum affinization of its vertical rather than horizontal subalgebra, with $\Uv$ in a Drinfeld-Jimbo presentation and $\Uh$ in a Drinfeld new presentation (although the multiplication is of course reversed).

\begin{cor}\label{Phi corollary}
    There is an automorphism $\Phi := \eta \psi$ of $\Utor$ with inverse $\Phi^{-1} = \eta \Phi \eta = \psi \eta$, determined by the conditions $\Phi v = h$ and $\Phi h = v \eta' \sigma$.
\end{cor}

\begin{rmk}
    \begin{enumerate}
        \item In type $A_{n}^{(1)}$ this is precisely the automorphism of Miki \cite{Miki99} with the extra deformation parameter $\kappa$ set to $1$.
        \item Note in particular the actions on central elements: $\psi$ exchanges $C$ and $(k_{0}^{a_{0}}\dots k_{n}^{a_{n}})^{-1}$, while
        $\Phi$ maps $C \mapsto k_{0}^{a_{0}}\dots k_{n}^{a_{n}}$ and $k_{0}^{a_{0}}\dots k_{n}^{a_{n}} \mapsto C^{-1}$.
    \end{enumerate}
\end{rmk}

The following proposition provides compatibilities between the action of $\Bdd$ on $\Utor$ and the (anti-)automorphisms $\psi$ and $\Phi^{\pm 1}$, which can therefore be viewed as quantum toroidal analogues of the corresponding involutions of the braid group.

\begin{prop}\label{compatibilities proposition}
    For all $b\in\Bdd$ we have
    \begin{itemize}
        \item $\psi \circ b = \te(b) \circ \psi, \hfill \refstepcounter{equation}(\theequation)\label{compatibility 1}$
        \item $\Phi \circ b = \gv(\te(b)) \circ \Phi, \hfill \refstepcounter{equation}(\theequation)\label{compatibility 2}$
        \item $\Phi^{-1} \circ b = \gh(\te(b)) \circ \Phi^{-1}. \hfill \refstepcounter{equation}(\theequation)\label{compatibility 3}$
    \end{itemize}
\end{prop}
\begin{proof}
    We start with the first identity.
    Note that since $\psi^{2} = \mathrm{id}$, having the relation for some $b\in\Bdd$ immediately implies it for $b^{-1}$ and $\te(b)$ as well.
    By Proposition~\ref{Ti properties} and Theorem~\ref{psi theorem},
    \begin{align*}
        \psi T_{i} h = v \sigma \Tb_{i} = v \Tb_{i}^{-1} \sigma = T_{i}^{-1} \psi h, \\
        \psi T_{i} v = h \sigma \Tb_{i} = h \Tb_{i}^{-1} \sigma = T_{i}^{-1} \psi v,
    \end{align*}
    for all $i\in I_{0}$, which gives the case $b = T_{i}$ since $\Uh$ and $\Uv$ generate $\Utor$.
    For $b = \pi_{j}$ we have that $\rho_{j} \psi (C^{\pm 1}) = \Cb^{\pm 1} = \psi \pi_{j} (C^{\pm 1})$ follows from
    $X_{\beta} \psi (C^{\pm 1}) = \Cb^{\pm 1} = \psi Y_{\beta} (C^{\pm 1})$ and the identity with $T_{1}^{\pm 1},\dots,T_{n}^{\pm 1}$.
    From (\ref{lemma (4)}),
    \begin{align*}
        &\rho_{j} \psi (\xpm_{i,m}) = \xbpm_{\rho_{j}(i),m} = \psi \pi_{j} (\xpm_{i,m}), \\
        &\rho_{j} \psi (k_{i}^{\pm 1}) = \kb_{\rho_{j}(i)}^{\pm 1} = \psi \pi_{j} (k_{i}^{\pm 1}),
    \end{align*}
    for all $i\in I$ and $m=0,\mp 1$, and therefore $\rho_{j} \psi = \psi \pi_{j}$.
    When $\Omega$ is non-trivial the proof is complete since $\B_{0}$, $\Omega$ and $\Omega^{v}$ generate $\Bdd$.
    In type $E_{8}^{(1)}$ it remains to check that $T_{0}\psi = \psi (T^{v}_{0})^{-1}$.
    First note that $T_{0} \psi v = h \Tb_{0} \sigma = h \sigma \Tb_{0}^{-1} = \psi (T^{v}_{0})^{-1} v$ so we are done on $\Uv$.
    Furthermore,
    \begin{align*}
        T_{0} \psi(k_{0})
        &= T_{0}(C^{-1} k_{1}^{a_{1}} \dots k_{n}^{a_{n}})
        = (C^{-1} k_{0}^{2} k_{1}^{a_{1}} \dots k_{n}^{a_{n}}) \\
        &= \psi (C^{-2} k_{0} k_{1}^{2a_{1}} \dots k_{n}^{2a_{n}})
        = \psi T_{s_{\theta}} (C^{-2} k_{0}) \\
        &= \psi T_{s_{\theta}} X_{-\theta^{\vee}} (k_{0})
        = \psi (T^{v}_{0})^{-1} (k_{0}),
    \end{align*}
    and for any $i\in I_{0}$, with $Y_{\mu'_{i}}$ as in the proof of Theorem~\ref{psi theorem}, we have
    \begin{align*}
        T_{0} \psi (\xb^{-}_{i,1})
        &= \sum_{s=0}^{-a_{0i}} (-1)^{s} q^{s} (\xm_{0,0})^{(s)} \xm_{i,1} (\xm_{0,0})^{(-a_{ij}-s)}  \\
        &= \psi \left( \sum_{s=0}^{-a_{0i}} (-1)^{s} q^{s} (\xb^{-}_{0,0})^{(s)} \xb^{-}_{i,1} (\xb^{-}_{0,0})^{(-a_{ij}-s)} \right) \\
        &= o(i) \psi Y_{\mu'_{i}} \left( \sum_{s=0}^{-a_{0i}} (-1)^{s} q^{s} (\xb^{-}_{0,0})^{(s)} \xb^{-}_{i,0} (\xb^{-}_{0,0})^{(-a_{ij}-s)} \right) \\
        &= o(i) \psi Y_{\mu'_{i}} (T^{v}_{0})^{-1} (\xb^{-}_{i,0})
        = o(i) \psi (T^{v}_{0})^{-1} Y_{\mu'_{i}} (\xb^{-}_{i,0}) \\
        &= (T^{v}_{0})^{-1} (\xb^{-}_{i,1}),
    \end{align*}
    and similarly for $\xb^{+}_{i,-1}$.
    This verifies that $T_{0}\psi = \psi (T^{v}_{0})^{-1}$ on $\Uh$ and so we have proved the first identity.
    It is immediate from the definitions that $\pi = \eta \pi \eta$ for each $\pi\in\Omega$ and $X_{\beta}^{-1} = \eta X_{\beta} \eta$ for any $\beta\in\Pov$.
    Furthermore, since $T_{i}^{-1} = \eta T_{i} \eta$ for all $i\in I$ we deduce the second identity from the first.
    The final identity then follows from $(\gv \circ \te)^{-1} = \gh \circ \te$.
\end{proof}

Our automorphism $\Phi$ should have a range of applications pertaining to the representation theory of simply laced $\Utor$.
For example, by conjugating the Drinfeld topological coproduct of $\Utor$ with $\Phi$ we hope to study tensor products of the type $1$ integrable loop-highest weight modules classified by Hernandez \cites{Hernandez05}, as well as their associated $R$-matrices, thus extending work of Miki \cites{Miki00} in type $A_{n}^{(1)}$.
\\

Another possible direction concerns the relationship between the level $(1,0)$ vertex representations of $\Utor$ due to Saito and Jing \cites{Saito98,Jing98}, and the type $1$ representations constructed geometrically by Nakajima \cites{Nak01,Nak02} using the equivariant K-theory of quiver varieties on the affine Dynkin diagrams.
The expectation is that twisting these vertex representations by $\Phi$ should be isomorphic to (the dual of) a Fock space representation.
\\

Of course, this result is well-known in type $A_{n}^{(1)}$ \cites{Miki00,Tsymbaliuk19} and moreover provides an instance of the celebrated boson-fermion correspondence from mathematical physics.
We note that in the other simply laced types, Fock space representations of $\Utor$ have only been described geometrically at this stage.
In particular, there does not yet exist a definition in terms of the $q$-wedge construction by Kashiwara-Miwa-Petersen-Yung \cites{KMPY96}.
\\

Nevertheless, such an identification would likely allow us to relate various geometric features on the quiver variety side with more algebraic or combinatorial elements on the other.
Indeed, interesting affine phenomena have already turned out to be meaningful in the geometric setting, for example in work of Nagao \cites{Nag09a,Nag09b} in type $A_{n}^{(1)}$.

\pagebreak

\appendix

\section{The affine Dynkin diagrams} 
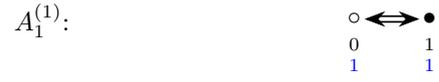
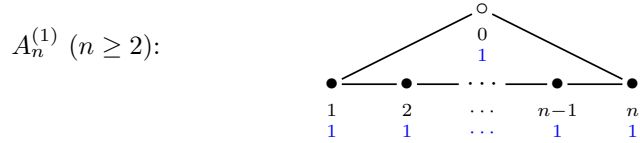
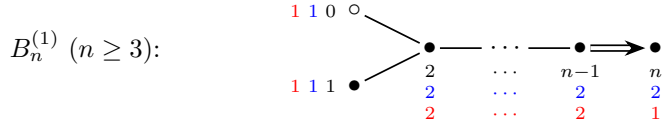
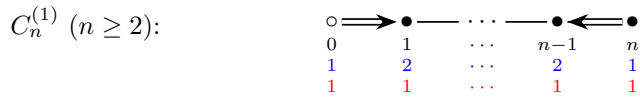
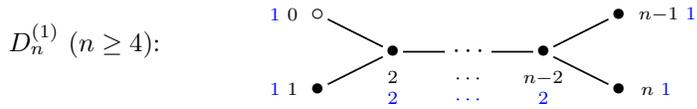
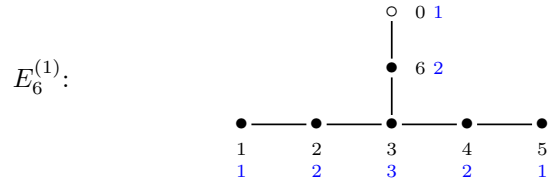
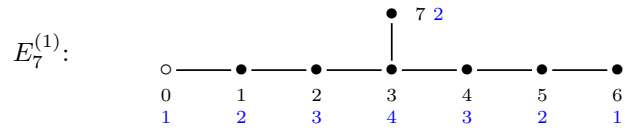
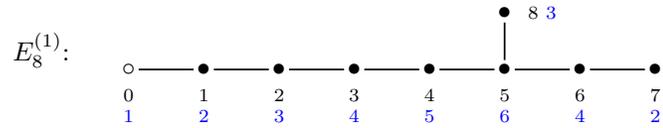
\label{Affine Dynkin Diagrams appendix}

\input{untwisted_diagrams}

\vspace{20pt}

\input{twisted_diagrams}

\pagebreak

\section{Proof of Lemma~\ref{lemma for new toroidal presentation}} \label{Proof of lemma}

Here we provide some details regarding the proof of Lemma~\ref{lemma for new toroidal presentation}.
Table~\ref{sequences} contains example sequences $(i_{1},i_{2},\dots,i_{h-1})$ in $I_{0}$ and $(\epsilon_{1},\dots,\epsilon_{h-2})$ in $\mathbb{Q}_{\leq 0}$ for $i_{1} = 1,2$, which allow us to write expressions for $\xhpm_{0,0}$ and $\hat{k}_{0}^{\pm 1}$.
\input{appendix_B_sequences}
\\

We first need to confirm that the Drinfeld-Jimbo relations of $U_{q}(X^{(1)})$ involving $\xpm_{0}$ and $k_{0}^{\pm 1}$ are preserved under the map $\xi : U_{q}(X^{(1)}) \rightarrow \A_{X}$ that sends $\xpm_{i} \mapsto \hat{x}^{\pm}_{i,0}$ and $t_{i}^{\pm 1} \mapsto \hat{k}_{i}^{\pm 1}$ for $i = 0,1,2$.
We outline a method for each below, referencing which relations \ref{AX 1}-\ref{AX 11} in $\A_{X}$ should be applied at each stage.
The author notes that many of the relations are easily checked with the help of a computer algebra package such as Magma.

\begin{itemize}[leftmargin=0pt]
\item Both $\kh_{0}^{\pm 1}\kh_{0}^{\mp 1} = 1$ and $[\kh_{0},\kh_{j}] = 0$ are trivial by \ref{AX 1}-\ref{AX 3}.
\item All $\kh_{0}\xhpm_{\ell,0}\kh_{0}^{-1} = q_{0}^{\pm a_{0\ell}}\xhpm_{\ell,0}$ and $\kh_{\ell}\xhpm_{0,0}\kh_{\ell}^{-1} = q_{0}^{\pm a_{\ell 0}}\xhpm_{0,0}$ are deduced from \ref{AX 1}-\ref{AX 4}.
\item The relation $[\xhp_{0,0},\xhm_{0,0}] = \frac{\kh_{0} - \kh_{0}^{-1}}{q_{0} - q_{0}^{-1}}$ is proved as follows:
\begin{enumerate}
    \item Input the expressions for $\xhpm_{0,0}$ with $i_{1} = 1$ and expand everything out.
    \item Factor the $\Ch^{\pm 1}$, $\kh_{1}^{\pm 1}$ and $\kh_{2}^{\pm 1}$ terms using \ref{AX 1} and \ref{AX 4}, then cancel them.
    \item Move any $\xhm_{i,m}$ terms in front of all $\xhp_{i,m}$ terms with \ref{AX 5}-\ref{AX 7}, then pull any $\Ch^{\pm 1}$, $\kh_{1}^{\pm 1}$ and $\kh_{2}^{\pm 1}$ factors created out to one side with \ref{AX 1} and \ref{AX 4}.
    \item Cancel all remaining summands other than $\frac{\kh_{0} - \kh_{0}^{-1}}{q_{0} - q_{0}^{-1}}$ using the relations \ref{AX 8} and \ref{AX 11} which involve $\xhpm_{1,0}$, $\xhpm_{1,\pm 1}$ and $\xhpm_{2,0}$.
\end{enumerate}
\item For $[\xhpm_{0,0},\xhmp_{\ell,0}] = 0$ with $\ell \in I_{0}$:
\begin{enumerate}
    \item Input the expression for $\xhpm_{0,0}$ with $i_{1} \not= \ell$ and expand everything out.
    \item Factor the $\Ch^{\pm 1}$, $\kh_{1}^{\pm 1}$ and $\kh_{2}^{\pm 1}$ terms using \ref{AX 1} and \ref{AX 4}.
    \item Cancel everything by \ref{AX 8} and \ref{AX 11} with $y_{i} = \xhmp_{\ell,0}$ and $y_{j} = \xhmp_{i_{1},0},\xhmp_{i_{1},\pm 1}$.
\end{enumerate}
\item For the quantum Serre relations between $\xhpm_{0,0}$ and $\xhpm_{\ell,0}$ with $\ell \in I_{0}$:
\begin{enumerate}
    \item Input the expression for $\xhpm_{0,0}$ with $i_{1} \not= \ell$ and expand everything out.
    \item Factor the $\Ch^{\pm 1}$, $\kh_{1}^{\pm 1}$ and $\kh_{2}^{\pm 1}$ terms using \ref{AX 1} and \ref{AX 4}.
    \item Move all $\xhpm_{\ell,0}$ terms to one side using relations \ref{AX 5} and \ref{AX 7}.
    \item Pull any $\kh_{\ell}^{\pm 1}$ factors created in the previous step out to one side with \ref{AX 4}.
    \item Cancel everything by \ref{AX 8}.
\end{enumerate}
\end{itemize}

Next we must give $\xpm_{1,\mp 1}$ and $\xpm_{2,\mp 1}$ in terms of the Drinfeld-Jimbo generators of $U_{q}(X^{(1)})$.
This can be done using Beck's extended affine braid group action from Theorem~\ref{Beck's affine result}.
In particular, writing each $X_{\omega_{i}^{\vee}}$ in the Coxeter presentation of $\Bd$ allows us to present its action with respect to the Drinfeld-Jimbo realization of $U_{q}(X^{(1)})$.
Then since $o(i)X_{\omega_{i}^{\vee}}$ sends $\xpm_{i,0}$ to $\xpm_{i,\mp 1}$ we can obtain the desired expressions, thus allowing us to find the images of $\xpm_{1,\mp 1}$ and $\xpm_{2,\mp 1}$ under $\xi$.
To complete the proof we check that these are equal to $\hat{x}^{\pm}_{1,\mp 1}$ and $\hat{x}^{\pm}_{2,\mp 1}$ respectively
by inserting the definitions of $\xhpm_{0,0}$ and $\kh_{0}^{\pm 1}$ in terms of the generators of $\A_{X}$ and applying the relevant relations.

\pagebreak
\addcontentsline{toc}{section}{References}

\input{bibliography}
\end{document}

%% file: Utor_illustration.tex
\begin{figure}
    \centering
\begin{tikzpicture}[scale=1]
    \node at (-0.9,0.35) {$\xpm_{0,0} ~~ k^{\pm 1}_{0}$};
    \node at (-0.9,1.5) {$\xpm_{0,1} ~~ h_{0,1}$};
    \node at (-0.9,-0.8) {$\xpm_{0,-1} ~~ h_{0,-1}$};
    \node at (-0.9,2.4) {$\vdots$};
    \node at (-0.9,-1.5) {$\vdots$};
    \node at (1.5,0.35) {$\xpm_{1,0} ~~ k^{\pm 1}_{1}$};
    \node at (1.5,1.5) {$\xpm_{1,1} ~~ h_{1,1}$};
    \node at (1.5,-0.8) {$\xpm_{1,-1} ~~ h_{1,-1}$};
    \node at (1.5,2.4) {$\vdots$};
    \node at (1.5,-1.5) {$\vdots$};
    \node at (3,-1.5) {$C^{\pm 1}$};
    \node at (3,0.35) {$\cdots$};
    \node at (3,1.5) {$\cdots$};
    \node at (3,-0.8) {$\cdots$};
    \node at (4.5,0.35) {$\xpm_{n,0} ~~ k^{\pm 1}_{n}$};
    \node at (4.5,1.5) {$\xpm_{n,1} ~~ h_{n,1}$};
    \node at (4.5,-0.8) {$\xpm_{n,-1} ~~ h_{n,-1}$};
    \node at (4.5,2.4) {$\vdots$};
    \node at (4.5,-1.5) {$\vdots$};
    \node at (-1.8,0.7) {$\color{blue} \Uh$};
    \node at (0.6,2.55) {$\color{red} \Uv$};
    \draw[draw=blue] (-2.1,-0.25) rectangle ++(7.8,1.2);
    \draw[draw=red] (0.3,-2.1) rectangle ++(5.5,4.9);
    \draw[draw=black] (-2.2,-2.2) rectangle ++(8.1,5.1);
\end{tikzpicture}
\caption{\hspace{.5em}Illustration of $\Utor$ and its quantum affine subalgebras $\Uh$ and $\Uv$}
\label{Utor illustration}
\end{figure}

%% file: Bdd_illustration.tex
\begin{figure}
    \centering
\begin{tikzpicture}[scale=1]
    \node at (-0.8,0.35) {$\Omega ~~ T_{0}^{\pm 1}$};
    \node at (1.5,0.35) {$T_{1}^{\pm 1}$};
    \node at (1.5,1.5) {$X_{\omega_{1}^{\vee}}$};
    \node at (1.5,-0.8) {$X_{-\omega_{1}^{\vee}}$};
    \node at (1.5,2.4) {$\vdots$};
    \node at (1.5,-1.5) {$\vdots$};
    \node at (3,0.35) {$\cdots$};
    \node at (3,1.5) {$\cdots$};
    \node at (3,-0.8) {$\cdots$};
    \node at (4.5,0.35) {$T_{n}^{\pm 1}$};
    \node at (4.5,1.5) {$X_{\omega_{n}^{\vee}}$};
    \node at (4.5,-0.8) {$X_{-\omega_{n}^{\vee}}$};
    \node at (4.5,2.4) {$\vdots$};
    \node at (4.5,-1.5) {$\vdots$};
    \node at (-1.8,0.7) {$\color{blue} \Bh$};
    \node at (0.6,2.55) {$\color{red} \Bv$};
    \draw[draw=blue] (-2.1,-0.25) rectangle ++(7.8,1.2);
    \draw[draw=red] (0.3,-2.1) rectangle ++(5.5,4.9);
    \draw[draw=black] (-2.2,-2.2) rectangle ++(8.1,5.1);
\end{tikzpicture}
\caption{\hspace{.5em}Illustration of $\Bdd$ and its extended affine braid subgroups $\Bh$ and $\Bv$}
\label{Bdd illustration}
\end{figure}

%% file: Bdd_notation_table.tex
\begin{table}
\centering
\begin{tabular}{ |c||c|c| }
 \hline
  & Coxeter generators & Bernstein generators \\
 \hline
 & & \\[-10.5pt]
 \hline
 $\Bh$
 &
 \begin{tabular}{@{}c@{}}$T_{1},\dots,T_{n}$ \\[1pt] $T_{0} = T_{s_{\theta}}^{-1} Y_{-\beta_{\theta}}$ \\[1pt] $\Omega = \lbrace \pi_{i} = Y_{\beta_{i}}T_{v_{i}^{-1}} : i\in\Imin \rbrace$\end{tabular}
 &
 \begin{tabular}{@{}c@{}}$T_{1},\dots,T_{n}$ \\[1pt] $\lbrace Y_{\mu} : \mu \in N \rbrace$\end{tabular}
 \\[15pt]
 \hline
& & \\[-11pt]
 $\Bv$
 &
 \begin{tabular}{@{}c@{}}$T_{1},\dots,T_{n}$ \\[1pt] $T^{v}_{0} = X_{\theta^{\vee}} T_{s_{\theta}}^{-1}$ \\[1pt] $\Omega^{v} = \lbrace \rho_{i} = X_{\omega_{i}^{\vee}}T_{v_{i}}^{-1} : i\in\Imin \rbrace$\end{tabular}
 &
 \begin{tabular}{@{}c@{}}$T_{1},\dots,T_{n}$ \\[1pt] $\lbrace X_{\beta} : \beta \in \Pov \rbrace$\end{tabular}
 \\[14pt]
 \hline
\end{tabular}
\caption{\hspace{.5em}Coxeter and Bernstein generators for $\Bh$ and $\Bv$}\label{extended double affine braid group table}
\end{table}

%% file: E8_table.tex
\begin{table}
\centering
\begin{tabular}{ |c||c|c| } 
 \hline
  $i$ & $Y_{\mu_{i}}$ & $Y_{\mu'_{i}}$ \\
 \hline
 & & \\[-11pt]
 \hline
 $1$ & $Y_{7}Y_{8}^{-1}$ & $Y_{1}Y_{7}^{-1}$ \\
 \hline
 $2$ & $Y_{8}Y_{6}^{-1}$ & $Y_{2}Y_{8}^{-1}$ \\
 \hline
 $3$ & $Y_{6}Y_{4}^{-1}$ & $Y_{3}Y_{6}^{-1}$ \\
 \hline
 $4$ & $Y_{7}Y_{8}Y_{5}^{-1}$ & $Y_{4}Y_{7}^{-1}Y_{8}^{-1}$ \\
 \hline
 $5$ & $Y_{8}^{2}Y_{4}^{-1}Y_{7}^{-1}$ & $Y_{5}Y_{8}^{-2}$ \\
 \hline
 $6$ & $Y_{3}Y_{4}^{-1}$ & $Y_{6}Y_{3}^{-1}$ \\
 \hline
 $7$ & $Y_{1}Y_{8}^{-1}$ & $Y_{7}Y_{1}^{-1}$ \\
 \hline
 $8$ & $Y_{2}Y_{6}^{-1}$ & $Y_{8}Y_{2}^{-1}$ \\
 \hline
\end{tabular}
\caption{\hspace{.5em}Elements $Y_{\mu_{i}},Y_{\mu'_{i}}\in\Bdd$ for each $i\in I_{0}$ in type $E_{8}^{(1)}$}\label{E8 Y table}
\end{table}

%% file: two_finite_Utor_illustrations.tex
\begin{figure}
    \centering
    \begin{tikzpicture}[scale=0.85, transform shape]
    \node at (-0.05,-0.35) {$k^{\pm 1}_{0}$};
    \node at (-0.05,0.35) {$\xpm_{0,0}$};
    \node at (-0.05,1.5) {$\xpm_{0,\mp 1}$};
    \node at (1.95,-0.35) {$k^{\pm 1}_{1}$};
    \node at (1.95,0.35) {$\xpm_{1,0}$};
    \node at (1.95,1.5) {$\xpm_{1,\mp 1}$};
    \node at (3,-1.5) {$C^{\pm 1}$};
    \node at (3,-0.35) {$\cdots$};
    \node at (3,0.35) {$\cdots$};
    \node at (3,1.5) {$\cdots$};
    \node at (4.05,-0.35) {$k^{\pm 1}_{n}$};
    \node at (4.05,0.35) {$\xpm_{n,0}$};
    \node at (4.05,1.5) {$\xpm_{n,\mp 1}$};
    \node at (-0.8,0.65) {$\color{blue} \Uh$};
    \node at (1.2,1.85) {$\color{red} \Uv$};
    \draw[draw=black] (-1.2,-2.2) rectangle ++(6.5,4.4);
    \draw[draw=blue] (-1.1,-0.9) rectangle ++(6.2,1.8);
    \draw[draw=red] (0.9,-2.1) rectangle ++(4.3,4.2);
    \end{tikzpicture}
    ~~~
    \begin{tikzpicture}[scale=0.85, transform shape]
    \node at (-0.5,0) {$\Cb^{\pm 1}$};
    \node at (0.3,0) {$\begin{array}{c}
        \xbpm_{1,\mp 1} \\
        \vdots \\
        \xbpm_{n,\mp 1}
    \end{array}$};
    \node at (-0.05,1.5) {$\xbpm_{0,\mp 1}$};
    \node at (1.95,-0.35) {$\kb^{\pm 1}_{1}$};
    \node at (1.95,0.35) {$\xbpm_{1,0}$};
    \node at (2.5,1.5) {$\xbpm_{0,0}$};
    \node at (3,-0.35) {$\cdots$};
    \node at (3,0.35) {$\cdots$};
    \node at (4.05,-0.35) {$\kb^{\pm 1}_{n}$};
    \node at (4.05,0.35) {$\xbpm_{n,0}$};
    \node at (3.5,1.5) {$\kb^{\pm 1}_{0}$};
    \node at (-0.8,0.65) {$\color{blue} \Uh$};
    \node at (1.2,1.85) {$\color{red} \Uv$};
    \draw[draw=black] (-1.2,-2.2) rectangle ++(6.5,4.4);
    \draw[draw=blue] (-1.1,-0.9) rectangle ++(6.2,1.8);
    \draw[draw=red] (0.9,-2.1) rectangle ++(4.3,4.2);
    \end{tikzpicture}
    \caption{\hspace{.5em}Illustrations of $\Utor$ displaying the two generating sets}\label{illustrations}
\end{figure}

%% file: untwisted_diagrams.tex
\begin{figure}[H]
\centering
\begin{tabular}{r c}
$A^{(1)}_{1}$:
&
\begin{tikzpicture}[baseline={([yshift=-.35cm]current bounding box.north)}, double distance = 1.5pt, line width=0.65pt, every node/.style={outer sep=-2pt}]
    \node (0) at (-2,-0.5) {$\circ$};
    \node (1) at (-1,-0.5) {$\bullet$};
    \node (a0) at (-2,-1) {\scriptsize
    $\begin{array}{c}
        \color{black} 0 \\
        \color{blue} 1
    \end{array}$ };
    \node (a1) at (-1,-1) {\scriptsize
    $\begin{array}{c}
        \color{black} 1 \\
        \color{blue} 1
    \end{array}$ };
    \draw[double, Stealth-Stealth] (0) -- (1);
\end{tikzpicture}
\\[1cm]
$A^{(1)}_{n}$ ($n\geq 2$):
&
\begin{tikzpicture}[baseline={([yshift=0.25cm]current bounding box.east)}, double distance = 1.5pt, line width=0.65pt, every node/.style={outer sep=-2pt}]
    \node (0) at (0,0) {$\circ$};
    \node (1) at (-2,-1) {$\bullet$};
    \node (2) at (-1,-1) {$\bullet$};
    \node (dots) at (0,-1) {$\,\cdots$};
    \node (n-1) at (1,-1) {$\bullet$};
    \node (n) at (2,-1) {$\bullet$};
    \node (a0) at (0,-0.5) {\scriptsize
    $\begin{array}{c}
        \color{black} 0 \\
        \color{blue} 1
    \end{array}$ };
    \node (a1) at (-2,-1.5) {\scriptsize
    $\begin{array}{c}
        \color{black} 1 \\
        \color{blue} 1
    \end{array}$ };
    \node (a2) at (-1,-1.5) {\scriptsize
    $\begin{array}{c}
        \color{black} 2 \\
        \color{blue} 1
    \end{array}$ };
    \node (dotslabels) at (0,-1.5) {\scriptsize
    $\begin{array}{c}
        \color{black} \cdots \\
        \color{blue} \cdots
    \end{array}$ };
    \node (an-1) at (1,-1.5) {\scriptsize
    $\begin{array}{c}
        \color{black} n\mathrm{-}1 \\
        \color{blue} 1
    \end{array}$ };
    \node (an) at (2,-1.5) {\scriptsize
    $\begin{array}{c}
        \color{black} n \\
        \color{blue} 1
    \end{array}$ };
    \draw (0) edge (1) (1) edge (2) (2) edge (dots) (dots) edge (n-1) (n-1) edge (n) (n) edge (0);
\end{tikzpicture}
\\[1cm]
$B^{(1)}_{n}$ ($n\geq 3$):
&
\begin{tikzpicture}[baseline={([yshift=0.1cm]current bounding box.east)}, double distance = 1.5pt, line width=0.65pt, every node/.style={outer sep=-2pt}]
    \node (0) at (-2,0) {$\circ$};
    \node (1) at (-2,-1) {$\bullet$};
    \node (2) at (-1,-0.5) {$\bullet$};
    \node (dots) at (0,-0.5) {$\,\cdots$};
    \node (n-1) at (1,-0.5) {$\bullet$};
    \node (n) at (2,-0.5) {$\bullet$};
    \node (a0) at (-2.55,0) {\scriptsize
    $\color{red} 1$
    $\color{blue} 1$
    $\color{black} 0$ };
    \node (a1) at (-2.55,-1) {\scriptsize
    $\color{red} 1$
    $\color{blue} 1$
    $\color{black} 1$ };
    \node (a2) at (-1,-1.1) {\scriptsize
    $\begin{array}{c}
        \color{black} 2 \\
        \color{blue} 2 \\
        \color{red} 2
    \end{array}$ };
    \node (dotslabels) at (0,-1.1) {\scriptsize
    $\begin{array}{c}
        \color{black} \cdots \\
        \color{blue} \cdots \\
        \color{red} \cdots
    \end{array}$ };
    \node (an-1) at (1,-1.1) {\scriptsize
    $\begin{array}{c}
        \color{black} n\mathrm{-}1 \\
        \color{blue} 2 \\
        \color{red} 2
    \end{array}$ };
    \node (an) at (2,-1.1) {\scriptsize
    $\begin{array}{c}
        \color{black} n \\
        \color{blue} 2 \\
        \color{red} 1
    \end{array}$ };
    \draw (0) edge (2) (1) edge (2) (2) edge (dots) (dots) edge (n-1);
    \draw[double, -Stealth] (n-1) -- (n);
\end{tikzpicture}
\\[1cm]
$C^{(1)}_{n}$ ($n\geq 2$):
&
\begin{tikzpicture}[baseline={([yshift=-.35cm]current bounding box.north)}, double distance = 1.5pt, line width=0.65pt, every node/.style={outer sep=-2pt}]
    \node (0) at (-2,-0.5) {$\circ$};
    \node (1) at (-1,-0.5) {$\bullet$};
    \node (dots) at (0,-0.5) {$\,\cdots$};
    \node (n-1) at (1,-0.5) {$\bullet$};
    \node (n) at (2,-0.5) {$\bullet$};
    \node (a0) at (-2,-1.1) {\scriptsize
    $\begin{array}{c}
        \color{black} 0 \\
        \color{blue} 1 \\
        \color{red} 1
    \end{array}$ };
    \node (a1) at (-1,-1.1) {\scriptsize
    $\begin{array}{c}
        \color{black} 1 \\
        \color{blue} 2 \\
        \color{red} 1
    \end{array}$ };
    \node (dotslabels) at (0,-1.1) {\scriptsize
    $\begin{array}{c}
        \color{black} \cdots \\
        \color{blue} \cdots \\
        \color{red} \cdots
    \end{array}$ };
    \node (an-1) at (1,-1.1) {\scriptsize
    $\begin{array}{c}
        \color{black} n\mathrm{-}1 \\
        \color{blue} 2 \\
        \color{red} 1
    \end{array}$ };
    \node (an) at (2,-1.1) {\scriptsize
    $\begin{array}{c}
        \color{black} n \\
        \color{blue} 1 \\
        \color{red} 1
    \end{array}$ };
    \draw (1) edge (dots) (dots) edge (n-1) (n-1) edge (n);
    \draw[double, -Stealth] (0) -- (1);
    \draw[double, Stealth-] (n-1) -- (n);
\end{tikzpicture}
\\[1cm]
$D^{(1)}_{n}$ ($n\geq 4$):
&
\begin{tikzpicture}[baseline={([yshift=0.1cm]current bounding box.east)}, double distance = 1.5pt, line width=0.65pt, every node/.style={outer sep=-2pt}]
    \node (0) at (-2,0) {$\circ$};
    \node (1) at (-2,-1) {$\bullet$};
    \node (2) at (-1,-0.5) {$\bullet$};
    \node (dots) at (0,-0.5) {$\,\cdots$};
    \node (n-2) at (1,-0.5) {$\bullet$};
    \node (n-1) at (2,0) {$\bullet$};
    \node (n) at (2,-1) {$\bullet$};
    \node (a0) at (-2.45,0) {\scriptsize
    $\color{blue} 1$
    $\color{black} 0$ };
    \node (a1) at (-2.45,-1) {\scriptsize
    $\color{blue} 1$
    $\color{black} 1$ };
    \node (a2) at (-1,-1) {\scriptsize
    $\begin{array}{c}
        \color{black} 2 \\
        \color{blue} 2
    \end{array}$ };
    \node (dotslabels) at (0,-1) {\scriptsize
    $\begin{array}{c}
        \color{black} \cdots \\
        \color{blue} \cdots
    \end{array}$ };
    \node (an-2) at (1,-1) {\scriptsize
    $\begin{array}{c}
        \color{black} n\mathrm{-}2 \\
        \color{blue} 2
    \end{array}$ };
    \node (an-1) at (2.65,0) {\scriptsize
    $\color{black} n\mathrm{-}1$
    $\color{blue} 1$ };
    \node (an) at (2.5,-1) {\scriptsize
    $\color{black} n$
    $\color{blue} 1$ };
    \draw (0) edge (2) (1) edge (2) (2) edge (dots) (dots) edge (n-2) (n-2) edge (n-1) (n-2) edge (n);
\end{tikzpicture}
\\[1cm]
$E^{(1)}_{6}$:
&
\begin{tikzpicture}[baseline={([yshift=0.1cm]current bounding box.east)}, double distance = 1.5pt, line width=0.65pt, every node/.style={outer sep=-2pt}]
    \node (0) at (0,0.5) {$\circ$};
    \node (1) at (-2,-1) {$\bullet$};
    \node (6) at (0,-0.25) {$\bullet$};
    \node (2) at (-1,-1) {$\bullet$};
    \node (3) at (0,-1) {$\bullet$};
    \node (4) at (1,-1) {$\bullet$};
    \node (5) at (2,-1) {$\bullet$};
    \node (a0) at (0.5,0.5) {\scriptsize
    $\color{black} 0$
    $\color{blue} 1$ };
    \node (a6) at (0.5,-0.25) {\scriptsize
    $\color{black} 6$
    $\color{blue} 2$ };
    \node (a1) at (-2,-1.5) {\scriptsize
    $\begin{array}{c}
        \color{black} 1 \\
        \color{blue} 1
    \end{array}$ };
    \node (a2) at (-1,-1.5) {\scriptsize
    $\begin{array}{c}
        \color{black} 2 \\
        \color{blue} 2
    \end{array}$ };
    \node (a3) at (0,-1.5) {\scriptsize
    $\begin{array}{c}
        \color{black} 3 \\
        \color{blue} 3
    \end{array}$ };
    \node (a4) at (1,-1.5) {\scriptsize
    $\begin{array}{c}
        \color{black} 4 \\
        \color{blue} 2
    \end{array}$ };
    \node (a5) at (2,-1.5) {\scriptsize
    $\begin{array}{c}
        \color{black} 5 \\
        \color{blue} 1
    \end{array}$ };
    \draw (0) edge (6) (6) edge (3) (1) edge (2) (2) edge (3) (3) edge (4) (4) edge (5);
\end{tikzpicture}
\\[1cm]
$E^{(1)}_{7}$:
&
\begin{tikzpicture}[baseline={([yshift=0.1cm]current bounding box.east)}, double distance = 1.5pt, line width=0.65pt, every node/.style={outer sep=-2pt}]
    \node (0) at (-3,-1) {$\circ$};
    \node (1) at (-2,-1) {$\bullet$};
    \node (7) at (0,-0.25) {$\bullet$};
    \node (2) at (-1,-1) {$\bullet$};
    \node (3) at (0,-1) {$\bullet$};
    \node (4) at (1,-1) {$\bullet$};
    \node (5) at (2,-1) {$\bullet$};
    \node (6) at (3,-1) {$\bullet$};
    \node (a0) at (-3,-1.5) {\scriptsize
    $\begin{array}{c}
        \color{black} 0 \\
        \color{blue} 1
    \end{array}$ };
    \node (a7) at (0.5,-0.25) {\scriptsize
    $\color{black} 7$
    $\color{blue} 2$ };
    \node (a1) at (-2,-1.5) {\scriptsize
    $\begin{array}{c}
        \color{black} 1 \\
        \color{blue} 2
    \end{array}$ };
    \node (a2) at (-1,-1.5) {\scriptsize
    $\begin{array}{c}
        \color{black} 2 \\
        \color{blue} 3
    \end{array}$ };
    \node (a3) at (0,-1.5) {\scriptsize
    $\begin{array}{c}
        \color{black} 3 \\
        \color{blue} 4
    \end{array}$ };
    \node (a4) at (1,-1.5) {\scriptsize
    $\begin{array}{c}
        \color{black} 4 \\
        \color{blue} 3
    \end{array}$ };
    \node (a5) at (2,-1.5) {\scriptsize
    $\begin{array}{c}
        \color{black} 5 \\
        \color{blue} 2
    \end{array}$ };
    \node (a6) at (3,-1.5) {\scriptsize
    $\begin{array}{c}
        \color{black} 6 \\
        \color{blue} 1
    \end{array}$ };
    \draw (0) edge (1) (7) edge (3) (1) edge (2) (2) edge (3) (3) edge (4) (4) edge (5) (5) edge (6);
\end{tikzpicture}
\\[1cm]
$E^{(1)}_{8}$:
&
\begin{tikzpicture}[baseline={([yshift=0.1cm]current bounding box.east)}, double distance = 1.5pt, line width=0.65pt, every node/.style={outer sep=-2pt}]
    \node (0) at (-3,-1) {$\circ$};
    \node (1) at (-2,-1) {$\bullet$};
    \node (8) at (2,-0.25) {$\bullet$};
    \node (2) at (-1,-1) {$\bullet$};
    \node (3) at (0,-1) {$\bullet$};
    \node (4) at (1,-1) {$\bullet$};
    \node (5) at (2,-1) {$\bullet$};
    \node (6) at (3,-1) {$\bullet$};
    \node (7) at (4,-1) {$\bullet$};
    \node (a0) at (-3,-1.5) {\scriptsize
    $\begin{array}{c}
        \color{black} 0 \\
        \color{blue} 1
    \end{array}$ };
    \node (a8) at (2.5,-0.25) {\scriptsize
    $\color{black} 8$
    $\color{blue} 3$ };
    \node (a1) at (-2,-1.5) {\scriptsize
    $\begin{array}{c}
        \color{black} 1 \\
        \color{blue} 2
    \end{array}$ };
    \node (a2) at (-1,-1.5) {\scriptsize
    $\begin{array}{c}
        \color{black} 2 \\
        \color{blue} 3
    \end{array}$ };
    \node (a3) at (0,-1.5) {\scriptsize
    $\begin{array}{c}
        \color{black} 3 \\
        \color{blue} 4
    \end{array}$ };
    \node (a4) at (1,-1.5) {\scriptsize
    $\begin{array}{c}
        \color{black} 4 \\
        \color{blue} 5
    \end{array}$ };
    \node (a5) at (2,-1.5) {\scriptsize
    $\begin{array}{c}
        \color{black} 5 \\
        \color{blue} 6
    \end{array}$ };
    \node (a6) at (3,-1.5) {\scriptsize
    $\begin{array}{c}
        \color{black} 6 \\
        \color{blue} 4
    \end{array}$ };
    \node (a7) at (4,-1.5) {\scriptsize
    $\begin{array}{c}
        \color{black} 7 \\
        \color{blue} 2
    \end{array}$ };
    \draw (0) edge (1) (8) edge (5) (1) edge (2) (2) edge (3) (3) edge (4) (4) edge (5) (5) edge (6) (6) edge (7);
\end{tikzpicture}
\\
\end{tabular}
\end{figure}

\begin{figure}[H]
\centering
\begin{tabular}{r c}
$F^{(1)}_{4}$:
&
\begin{tikzpicture}[baseline={([yshift=-.35cm]current bounding box.north)}, double distance = 1.5pt, line width=0.65pt, every node/.style={outer sep=-2pt}]
    \node (0) at (-2,-0.5) {$\circ$};
    \node (1) at (-1,-0.5) {$\bullet$};
    \node (2) at (0,-0.5) {$\bullet$};
    \node (3) at (1,-0.5) {$\bullet$};
    \node (4) at (2,-0.5) {$\bullet$};
    \node (a0) at (-2,-1.1) {\scriptsize
    $\begin{array}{c}
        \color{black} 0 \\
        \color{blue} 1 \\
        \color{red} 1
    \end{array}$ };
    \node (a1) at (-1,-1.1) {\scriptsize
    $\begin{array}{c}
        \color{black} 1 \\
        \color{blue} 2 \\
        \color{red} 2
    \end{array}$ };
    \node (a2) at (0,-1.1) {\scriptsize
    $\begin{array}{c}
        \color{black} 2 \\
        \color{blue} 3 \\
        \color{red} 3
    \end{array}$ };
    \node (a3) at (1,-1.1) {\scriptsize
    $\begin{array}{c}
        \color{black} 3 \\
        \color{blue} 4 \\
        \color{red} 2
    \end{array}$ };
    \node (a4) at (2,-1.1) {\scriptsize
    $\begin{array}{c}
        \color{black} 4 \\
        \color{blue} 2 \\
        \color{red} 1
    \end{array}$ };
    \draw (0) edge (1) (1) edge (2) (3) -- (4);
    \draw[double, -Stealth] (2) -- (3);
\end{tikzpicture}
\\[1cm]
$G^{(1)}_{2}$:
&
\begin{tikzpicture}[baseline={([yshift=-.35cm]current bounding box.north)}, double distance = 1.5pt, line width=0.65pt, every node/.style={outer sep=-2pt}]
    \node (0) at (-2,-0.5) {$\circ$};
    \node (1) at (-1,-0.5) {$\bullet$};
    \node (2) at (0,-0.5) {$\bullet$};
    \node (a0) at (-2,-1.1) {\scriptsize
    $\begin{array}{c}
        \color{black} 0 \\
        \color{blue} 1 \\
        \color{red} 1
    \end{array}$ };
    \node (a1) at (-1,-1.1) {\scriptsize
    $\begin{array}{c}
        \color{black} 1 \\
        \color{blue} 2 \\
        \color{red} 2
    \end{array}$ };
    \node (a2) at (0,-1.1) {\scriptsize
    $\begin{array}{c}
        \color{black} 2 \\
        \color{blue} 3 \\
        \color{red} 1
    \end{array}$ };
    \draw (0) edge (1);
    \draw[-{Stealth[scale=1.5]}] (1) edge (2);
    \draw (-0.865,-0.43) edge (-0.35,-0.43);
    \draw (-0.865,-0.57) edge (-0.35,-0.57);
\end{tikzpicture}
\\
\end{tabular}
\caption{\hspace{.5em}The untwisted affine Dynkin diagrams.
Black labels are vertex numbers,
blue labels are $a_{i}$ values,
and in the non-symmetric cases $a_{i}^{\vee}$ values are in red}
\label{fig:Untwisted_affine_Dynkin_diagrams}
\end{figure}

%% file: twisted_diagrams.tex
\begin{figure}[H]
\centering
\begin{tabular}{r c}
$A^{(2)}_{2}$:
&
\begin{tikzpicture}[baseline={([yshift=-.35cm]current bounding box.north)}, double distance = 1.5pt, line width=0.65pt, every node/.style={outer sep=-2pt}]
    \node (0) at (-2,-0.5) {\large $\circ$};
    \node (1) at (-1,-0.5) {\large $\bullet$};
    \node (a0) at (-2,-1.1) {\scriptsize
    $\begin{array}{c}
        \color{black} 0 \\
        \color{blue} 2 \\
        \color{red} 1
    \end{array}$ };
    \node (a1) at (-1,-1.1) {\scriptsize
    $\begin{array}{c}
        \color{black} 1 \\
        \color{blue} 1 \\
        \color{red} 2
    \end{array}$ };
    \draw[white, {Stealth[scale=1.5, color=black, fill=black]}-] (0) edge (1);
    \draw (-1.6,-0.42) edge (-1.135,-0.42);
    \draw (-1.75,-0.475) edge (-1.135,-0.475);
    \draw (-1.75,-0.525) edge (-1.135,-0.525);
    \draw (-1.6,-0.58) edge (-1.135,-0.58);
\end{tikzpicture}
\\[1.2cm]
$A^{(2)}_{2n}$ ($n\geq 2$):
&
\begin{tikzpicture}[baseline={([yshift=-.35cm]current bounding box.north)}, double distance = 1.5pt, line width=0.65pt, every node/.style={outer sep=-2pt}]
    \node (0) at (-2,-0.5) {\large $\circ$};
    \node (1) at (-1,-0.5) {\large $\bullet$};
    \node (dots) at (0,-0.5) {$\,\cdots$};
    \node (n-1) at (1,-0.5) {\large $\bullet$};
    \node (n) at (2,-0.5) {\large $\bullet$};
    \node (a0) at (-2,-1.1) {\scriptsize
    $\begin{array}{c}
        \color{black} 0 \\
        \color{blue} 2 \\
        \color{red} 1
    \end{array}$ };
    \node (a1) at (-1,-1.1) {\scriptsize
    $\begin{array}{c}
        \color{black} 1 \\
        \color{blue} 2 \\
        \color{red} 2
    \end{array}$ };
    \node (dotslabels) at (0,-1.1) {\scriptsize
    $\begin{array}{c}
        \color{black} \cdots \\
        \color{blue} \cdots \\
        \color{red} \cdots
    \end{array}$ };
    \node (an-1) at (1,-1.1) {\scriptsize
    $\begin{array}{c}
        \color{black} n\mathrm{-}1 \\
        \color{blue} 2 \\
        \color{red} 2
    \end{array}$ };
    \node (an) at (2,-1.1) {\scriptsize
    $\begin{array}{c}
        \color{black} n \\
        \color{blue} 1 \\
        \color{red} 2
    \end{array}$ };
    \draw (1) edge (dots) (dots) edge (n-1) (n-1) edge (n);
    \draw[double, Stealth-] (0) -- (1);
    \draw[double, Stealth-] (n-1) -- (n);
\end{tikzpicture}
\\[1.2cm]
$A^{(2)}_{2n-1}$ ($n\geq 3$):
&
\begin{tikzpicture}[baseline={([yshift=0.1cm]current bounding box.east)}, double distance = 1.5pt, line width=0.65pt, every node/.style={outer sep=-2pt}]
    \node (0) at (-2,0) {\large $\circ$};
    \node (1) at (-2,-1) {\large $\bullet$};
    \node (2) at (-1,-0.5) {\large $\bullet$};
    \node (dots) at (0,-0.5) {$\,\cdots$};
    \node (n-1) at (1,-0.5) {\large $\bullet$};
    \node (n) at (2,-0.5) {\large $\bullet$};
    \node (a0) at (-2.55,0) {\scriptsize
    $\color{red} 1$
    $\color{blue} 1$
    $\color{black} 0$ };
    \node (a1) at (-2.55,-1) {\scriptsize
    $\color{red} 1$
    $\color{blue} 1$
    $\color{black} 1$ };
    \node (a2) at (-1,-1.1) {\scriptsize
    $\begin{array}{c}
        \color{black} 2 \\
        \color{blue} 2 \\
        \color{red} 2
    \end{array}$ };
    \node (dotslabels) at (0,-1.1) {\scriptsize
    $\begin{array}{c}
        \color{black} \cdots \\
        \color{blue} \cdots \\
        \color{red} \cdots
    \end{array}$ };
    \node (an-1) at (1,-1.1) {\scriptsize
    $\begin{array}{c}
        \color{black} n\mathrm{-}1 \\
        \color{blue} 2 \\
        \color{red} 2
    \end{array}$ };
    \node (an) at (2,-1.1) {\scriptsize
    $\begin{array}{c}
        \color{black} n \\
        \color{blue} 1 \\
        \color{red} 2
    \end{array}$ };
    \draw (0) edge (2) (1) edge (2) (2) edge (dots) (dots) edge (n-1);
    \draw[double, Stealth-] (n-1) -- (n);
\end{tikzpicture}
\\[1.2cm]
$D^{(2)}_{n+1}$ ($n\geq 2$):
&
\begin{tikzpicture}[baseline={([yshift=-.35cm]current bounding box.north)}, double distance = 1.5pt, line width=0.65pt, every node/.style={outer sep=-2pt}]
    \node (0) at (-2,-0.5) {\large $\circ$};
    \node (1) at (-1,-0.5) {\large $\bullet$};
    \node (dots) at (0,-0.5) {$\,\cdots$};
    \node (n-1) at (1,-0.5) {\large $\bullet$};
    \node (n) at (2,-0.5) {\large $\bullet$};
    \node (a0) at (-2,-1.1) {\scriptsize
    $\begin{array}{c}
        \color{black} 0 \\
        \color{blue} 1 \\
        \color{red} 1
    \end{array}$ };
    \node (a1) at (-1,-1.1) {\scriptsize
    $\begin{array}{c}
        \color{black} 1 \\
        \color{blue} 1 \\
        \color{red} 2
    \end{array}$ };
    \node (dotslabels) at (0,-1.1) {\scriptsize
    $\begin{array}{c}
        \color{black} \cdots \\
        \color{blue} \cdots \\
        \color{red} \cdots
    \end{array}$ };
    \node (an-1) at (1,-1.1) {\scriptsize
    $\begin{array}{c}
        \color{black} n\mathrm{-}1 \\
        \color{blue} 1 \\
        \color{red} 2
    \end{array}$ };
    \node (an) at (2,-1.1) {\scriptsize
    $\begin{array}{c}
        \color{black} n \\
        \color{blue} 1 \\
        \color{red} 1
    \end{array}$ };
    \draw (1) edge (dots) (dots) edge (n-1) (n-1) edge (n);
    \draw[double, Stealth-] (0) -- (1);
    \draw[double, -Stealth] (n-1) -- (n);
\end{tikzpicture}
\\[1.2cm]
$E^{(2)}_{6}$:
&
\begin{tikzpicture}[baseline={([yshift=-.35cm]current bounding box.north)}, double distance = 1.5pt, line width=0.65pt, every node/.style={outer sep=-2pt}]
    \node (0) at (-2,-0.5) {\large $\circ$};
    \node (1) at (-1,-0.5) {\large $\bullet$};
    \node (2) at (0,-0.5) {\large $\bullet$};
    \node (3) at (1,-0.5) {\large $\bullet$};
    \node (4) at (2,-0.5) {\large $\bullet$};
    \node (a0) at (-2,-1.1) {\scriptsize
    $\begin{array}{c}
        \color{black} 0 \\
        \color{blue} 1 \\
        \color{red} 1
    \end{array}$ };
    \node (a1) at (-1,-1.1) {\scriptsize
    $\begin{array}{c}
        \color{black} 1 \\
        \color{blue} 2 \\
        \color{red} 2
    \end{array}$ };
    \node (a2) at (0,-1.1) {\scriptsize
    $\begin{array}{c}
        \color{black} 2 \\
        \color{blue} 3 \\
        \color{red} 3
    \end{array}$ };
    \node (a3) at (1,-1.1) {\scriptsize
    $\begin{array}{c}
        \color{black} 3 \\
        \color{blue} 2 \\
        \color{red} 4
    \end{array}$ };
    \node (a4) at (2,-1.1) {\scriptsize
    $\begin{array}{c}
        \color{black} 4 \\
        \color{blue} 1 \\
        \color{red} 2
    \end{array}$ };
    \draw (0) edge (1) (1) edge (2) (3) -- (4);
    \draw[double, Stealth-] (2) -- (3);
\end{tikzpicture}
\\[1.2cm]
$D^{(3)}_{4}$:
&
\begin{tikzpicture}[baseline={([yshift=-.35cm]current bounding box.north)}, double distance = 1.5pt, line width=0.65pt, every node/.style={outer sep=-2pt}]
    \node (0) at (-2,-0.5) {\large $\circ$};
    \node (1) at (-1,-0.5) {\large $\bullet$};
    \node (2) at (0,-0.5) {\large $\bullet$};
    \node (a0) at (-2,-1.1) {\scriptsize
    $\begin{array}{c}
        \color{black} 0 \\
        \color{blue} 1 \\
        \color{red} 1
    \end{array}$ };
    \node (a1) at (-1,-1.1) {\scriptsize
    $\begin{array}{c}
        \color{black} 1 \\
        \color{blue} 2 \\
        \color{red} 2
    \end{array}$ };
    \node (a2) at (0,-1.1) {\scriptsize
    $\begin{array}{c}
        \color{black} 2 \\
        \color{blue} 1 \\
        \color{red} 3
    \end{array}$ };
    \draw (0) edge (1);
    \draw[{Stealth[scale=1.5]}-] (1) edge (2);
    \draw (-0.6,-0.43) edge (-0.151,-0.43);
    \draw (-0.6,-0.57) edge (-0.151,-0.57);
\end{tikzpicture}
\\
\end{tabular}
\caption{\hspace{.5em}The twisted affine Dynkin diagrams.
Black labels are vertex numbers,
blue labels are $a_{i}$ values,
and red labels are $a_{i}^{\vee}$ values}
\label{fig:Twisted_affine_Dynkin_diagrams}
\end{figure}

%% file: appendix_B_sequences.tex
\begin{table}
\centering
\begin{tabular}{|c|c|c|}
\hline
Type $A_{2}$ & Type $C_{2}$ & Type $G_{2}$ \\
\hline
& & \\[-11pt]
\hline
$\alpha_{1} \xrightarrow{-1} \alpha_{2}$
&
$\alpha_{1} \xrightarrow{-1} \alpha_{2} \xrightarrow{0} \alpha_{1}$
&
$\alpha_{1} \xrightarrow{-1} \alpha_{2} \xrightarrow{-1/3} \alpha_{2} \xrightarrow{0} \alpha_{1} \xrightarrow{-2/3} \alpha_{2}$ \\
$\alpha_{2} \xrightarrow{-1} \alpha_{1}$
&
$\alpha_{2} \xrightarrow{-1} \alpha_{1} \xrightarrow{0} \alpha_{1}$
&
$\alpha_{2} \xrightarrow{-1} \alpha_{1} \xrightarrow{-1/3} \alpha_{2} \xrightarrow{0} \alpha_{1} \xrightarrow{-2/3} \alpha_{2}$ \\
\hline
\end{tabular}
\caption{\hspace{.5em}Sequences $\alpha_{i_{1}} \xrightarrow{\epsilon_{1}} \cdots \xrightarrow{\epsilon_{h-2}} \alpha_{i_{h-1}}$ for $i_{1} = 1,2$ in types $A_{2}$, $C_{2}$ and $G_{2}$}\label{sequences}
\end{table}